\documentclass[11pt, reqno]{amsart}
\usepackage{amssymb,amsmath,amsthm,amstext,amscd,latexsym,graphics,graphicx,bbm,caption,enumerate}

\usepackage[dvipsnames,svgnames,table]{xcolor}
\usepackage{mathtools}
\usepackage{hyperref}
\usepackage{tikz-cd}
\usetikzlibrary{matrix,arrows,decorations.pathmorphing}
\usetikzlibrary{positioning}
\usepackage{relsize}
\usepackage{enumerate}
\usepackage[shortlabels]{enumitem}
\hypersetup{plainpages=false,colorlinks=true, pagebackref}

\newtheorem{intro-thm}{Theorem}[]
\theoremstyle{plain}
\newtheorem{thm}{Theorem}[section]
\newtheorem{theorem}[thm]{Theorem}

\newtheorem{lemma}[thm]{Lemma}
\newtheorem{corollary}[thm]{Corollary}
\newtheorem{proposition}[thm]{Proposition}

\theoremstyle{definition}
\newtheorem{remark}[thm]{Remark}

\newtheorem{point}[thm]{}
\newtheorem{notation}[thm]{Notations}

\newtheorem{definition}[thm]{Definition}

\newtheorem{example}[thm]{Example}

\newtheorem{construction}[thm]{Construction}

\newcommand{\Pic}{{\rm Pic}}

\newcommand{\Div}{{\rm Div}}

\newcommand{\Aut}{{\rm Aut}}

\newcommand{\pr}{\mathrm{pr}}

\newcommand{\Spec}{{\rm Spec \,}}

\renewcommand{\tilde}{\widetilde}

\newcommand{\sC}{{\mathcal C}}

\newcommand{\sF}{{\mathcal F}}
\newcommand{\sG}{{\mathcal G}}
\newcommand{\sH}{{\mathcal H}}

\newcommand{\sL}{{\mathcal L}}

\newcommand{\sO}{{\mathcal O}}

\newcommand{\sS}{{\mathcal S}}

\newcommand{\sX}{{\mathcal X}}
\newcommand{\sY}{{\mathcal Y}}

\newcommand{\A}{{\mathbb A}}

\newcommand{\C}{{\mathbb C}}

\newcommand{\G}{{\mathbb G}}

\renewcommand{\P}{{\mathbb P}}

\newcommand{\Z}{{\mathbb Z}}

\input{xy}
\xyoption{all}
\title{$\mathbb{A}^1$-fibration in algebraic geometry and $\mathbb{A}^1$-homotopy type.
}

\begin{document}

\author{Utsav Choudhury}

\author{Aritra Mandal}
	
\author{Biman Roy}
\begin{abstract}
	In this article we show that an $\mathbb{A}^1$ bundle map or a vector bundle map $p: X \to Y$ induces trivial local fibration $\underline{Sing}(X) \to \underline{Sing}(Y)$. Using this, we first show that for Korus Russel threefolds of first kind $X$ the space  $\underline{Sing}(X)$ is $\mathbb{A}^1$  local. Then we show that for any smooth affine complex surface $X$, the $\mathbb{A}^1$- connected component sheaf is homotopy invariant. 
\end{abstract}

\maketitle

\section{Introduction}	
The homotopy information of a manifold $X$ fibered over a base space $Y$ with a fixed fiber $F$ can be recovered from the homotopy information of both $Y$ and $F$. This is because a 
fibration of the form $F \to X \to Y$, gives rise to the following exact sequence of homotopy groups (and pointed sets for $\pi_0$) : 
\begin{align*}
	\dots \to \pi_n(F) \to \pi_n(X) \to \pi_n(Y) \to & \pi_{n-1}(F) \to \dots  \\
	&\to \pi_0(F) \to \pi_0(X) \to \pi_0(Y).
\end{align*}

\ 

If the fiber $F$ is moreover contractible (i.e. $\pi_n(F) = 0$ for all $n \geq 0$), a direct application of the long exact sequence or alternatively, the Mayer-Vietoris property for local trivializations demonstrates that the homotopy invariants of the base $Y$ are isomorphic to those of the total space $X$. Such fibrations were classically exploited to compute homotopy groups of spheres and classify fiber bundles (see \cite{steenrod} or \cite{serre thesis}).

\

In the setting of algebraic geometry, one encounters a different class of fibrations. Consider a morphism $p : X \to Y$ of algberaic varieties that is generically trivial with a given variety $F$ as its generic fiber. Consequently, the ``bad" (or singular) fibers are confined to a closed locus of codimension one or greater. In many cases where the base and fibers are smooth, the rigidity inherent in algebraic geometry ensures that the topology of these degenerate fibers is heavily constrained by the generic fiber. For instance, when the base $Y$ is a curve and the generic fiber is an affine curve, bad fibers occur only over finitely many closed points. Furthermore, the Euler characteristic of the components of these bad fibers is explicitly determined by the Euler characteristic of the generic fiber. (For $\mathbb{A}^1$-fibrations on surfaces and their degeneration, see \cite{miyanishicrm} or  \cite{kambayashi}. For more general fiber degenerations and Euler characteristics, see \cite{iitaka}.)

\ 

In such a situation, it is desirable that the homotopy invariants of the total space be recoverable from the homotopy information of the generic fiber, the homotopy information of the base, and the geometry/configuration of the degenerate (bad) fibers. This is readily achieved for vector bundles and, more generally, vector bundle torsors (i.e., affine spaces bundles). These represent the precise cases where the generic and degenerate fibers coincide structurally, thereby forming a genuine fiber bundle (see, for example, \cite{brauer group}). Consequently, non-trivial phenomena arise when genuine degenerations and bad fibers occur.

\

For instance, the complex sphere $S^2_{\mathbb{C}}$ admits an $\mathbb{A}^1$-fibration over $\mathbb{A}^1$, where the singular fiber is two disjoint copies of $\mathbb{A}^1$. Classically, it is well known that its underlying topological homotopy type is that of $\mathbb{CP}^1$,  which is also equivalent to the homotopy type of $\mathbb{A}^1$ with doubled origin. This can also be deduced from the fact that $S^2_{\mathbb{C}}$ carries the structure of an $\mathbb{A}^1$-bundle over $\mathbb{CP}^1$. The topological and birational classification of smooth affine surafces with negative log Kodaira dimension, which naturally exhibit such $\mathbb{A}^1$-fibrations, has been studied extensively in the literature (see, \cite{miyanishicrm}).

\

In this article, we explore the $\mathbb{A}^1$-algebraic topology of varieties admitting an $\mathbb{A}^1$-fibration over curves or surfaces. Any $\mathbb{A}^1$-fibration $X \to C$ over a curve $C$ has $\mathbb{A}^1$ as its generic fiber, while its bad fibers can consist at most of disjoint unions of "thickened" $\mathbb{A}^1$'s ( non-reduced lines). Within the framework of 
$\mathbb{A}^1$-algebraic topology (as developed by Morel and Voevodsky \cite{mv}), the affine line $\mathbb{A}^1$ is contractible to a point. Generically, therefore, the $\mathbb{A}^1$-homotopy type of $X$ is equivalent to the $\mathbb{A}^1$-homotopy type of $C$. 

\

It is natural to expect that the global $\mathbb{A}^1$-homotopy type of $X$ can be fully recovered from the base $C$ alongside the geometry of the bad fibers - specifically, the number of singular fibers, their component count, and their respective multiplicities. In fact, the fibration $X \to C$ can be factored through a smooth algebraic space $\mathcal{C}$ that is birational to $C$, such that the induced map $X \to \mathcal{C}$ is an étale $\mathbb{A}^1$-bundle. The space $\mathcal{C}$ obtained via this factorization precisely encodes the data of the bad points, the number of components of the fibers over them, and their corresponding multiplicities \cite{doub}. 

\

We first prove that (Theorem \ref{sing-etale-bundle}) if   $\rho: X \to \mathcal{C}$ is an étale locally trivial $\mathbb{A}^1$-bundle, such that $X$ is a smooth variety over $k$ and $\mathcal{C}$ is a smooth algebraic space over $k$, then the induced morphism $\rho_*: Sing_*(X) \to Sing_*(\mathcal{C})$ is a Nisnevich local weak equivalence. Note that, when $\mathcal{C}$ is a smooth scheme and $\rho$ is a Nisnevich $\mathbb{A}^n$-bundle then this result obviously follows from homotopy invariance and Nisnevich Mayer Vietoris \cite{krishnakumar}.
Our theorem (Theorem \ref{sing-etale-bundle}) has stronger implications.

\

For instance, if $\rho : X \to \mathcal{C}$ as in the Theorem  \ref{sing-etale-bundle}, we get  $\sS(X) \cong \sS(\mathcal{C})$ as Nisnevich sheaves (see \cite[Definition 2.9]{bhs} for the construction of $\sS(X)$). This observation, together with rigidity property of the base curve and techniques of ghost homotopy, is used to prove Theorem \ref{connected-base-new} and Theorem \ref{rigid-base}. These prove that for any $\mathbb{A}^1$-fibration $X \to C$ over a smooth curve $C$ we have $\pi_0^{\mathbb{A}^1}(X)$ is $\mathbb{A}^1$-invariant. A corollary of these theorems is  that for all smooth affine complex surfaces $X$ of negative log Kodaira dimesnion, $\pi_0^{\mathbb{A}^1}(X)$ is $\mathbb{A}^1$-invariant (Corollary \ref{negative kodaira invariant}).
Using similar techniques as in \cite[Proposition 2.9]{choudhury2}, 
we get that for non-negative log Kodaira dimension smooth affine complex surfaces $X$, $\sS(X) = \sS^2(X)$ and therefore $\pi_0^{\mathbb{A}^1}(X)$ is homotopy invariant for all smooth affine complex surfaces. Homotopy invariance of $\pi_0^{\mathbb{A}^1}(X)$ of smooth proper complex surfaces has been proven in \cite[Corollary 3.15]{bhs} and \cite[Theorem 1.2]{bs}.

\

Another direct consequence of our Theorem \ref{sing-etale-bundle} is that $Sing_*(X)$ is $\mathbb{A}^1$-local if and only if $Sing_*(\mathcal{C})$ is $\mathbb{A}^1$-local. The property of $Sing_*(S)$ for space $S$ being $\mathbb{A}^1$-local is extremely useful (see \cite{am}, \cite{asokI}, \cite{asokII}, \cite{balwe reductive}).
For instance if $S = \mathbb{A}^2 \setminus \left\{(0,0)\right\}$ \cite[Theorem 8.9]{mor} (see also \cite[Corollary 4.2.6]{asokII})
or $S = \mathbb{A}^n$, then $Sing_*(S)$ is $\mathbb{A}^1$-local. If $Y$ is a retract of a space $S$, such that $Sing_*(S)$ $\mathbb{A}^1$-local, then $Sing_*(Y)$ is $\mathbb{A}^1$-local. Therefore if a variety $Y$ is retract of $\mathbb{A}^n$ then $Sin_*(Y)$ has to be $\mathbb{A}^1$-local.  

\

We show that if $X$ is a Koras-Russell threefold of the first kind then $Sing_*(X)$ is $\mathbb{A}^1$-local (Theorem \ref{application purity koras}). We use exactly the same (as in 
\cite{df}) using weak unstable five lemma applied to the same diagram of cofibration. It is also useful to note that for all $\mathbb{A}^1$-fibration $\rho : X \to C$ with $X$ smooth affine complex surface  such that the algebraic space curve $\mathcal{C}$ is $\mathbb{A}^1$-rigid we get $Sing_*(X)$ is $\mathbb{A}^1$-local (Corollary \ref{naive spaces}).

\

The remainder of this article is organized as follows. In Section 2,  we briefly recall the construction of unstable motivic homotopy category and review the core properties which are relevant to our main results. In particular, we prove that the map  $Sing_*(\mathbb{A}^n) \to Spec(k)$ is a sectionwise trivial fibration (Example \ref{kan fibrant An}, Appendix add general horn filling formula and refer here, Lemma \ref{contractible An}). Finally, we establish the main result,  Theorem \ref{sing-etale-bundle}.

\

In Section 3,  we prove Lemma \ref{key lemma} which shows that $Sing_*$ commutes with good quotient.  Using \ref{sing-etale-bundle}, purity, excision and weak five lemma,  we prove
Theorem  \ref{application purity koras}.

\

The first important result in Section 4 is Proposition \ref{ghost homotopy implies dominant 1}. This shows that existence of nontrivial ghost homotopy on a surface $X$ implies that the surface $X$ is log-uniruled. The method is similar to the method used in (\cite{cb} \cite{choudhury2}).
This proposition is used to prove that for smooth affine complex surfaces $X$ of log Kodaira dimension non negative, we have $S(X) = S^2(X)$ and therefore $\pi_0^{\mathbb{A}^1}(X)$ is $\mathbb{A}^1$-invariant (Theorem \ref{non-uniruled}, Corollary \ref{invariant non-uniruled}). Finally, we use Theorem \ref{rigid-base} Theorem \ref{connected-base-new} to conclude that for any smooth affine complex surface $X$, $\pi_0^{\mathbb{A}^1}(X)$ is $\mathbb{A}^1$-invariant. 

\

Section 5 is dedicated to computing the $\mathbb{A}^1$-homotopy types of surfaces with negative log Kodaira dimension. In particular, we provide a complete classification of the $\mathbb{A}^1$-homotopy types for $\mathbb{A}^1$-fibrations $X \to C$ whose fiber components contain no non-reduced (thick) $\mathbb{A}^1$'s. Finally, the Appendix contains several explicit computations related to the horn-filling property; since these results lack a standard reference in the literature, we have included them here for completeness.

\

\textbf{Acknowledgements} :

The authors would like to thank Adrien Dubouloz and Anand Sawant for their helpful comments and suggestions. This article constitutes part of the Ph.D. thesis of the second author, who expresses deep gratitude to the Indian Statistical Institute for providing excellent resources and support throughout this work and to Aranya Kumar Bal for stimulating discussion about $\A^1$-fibration to $\A^1$, especially for Example \ref{rigid example}. The first author was supported by the ANRF/ARGM; this work is part of the project ANRF/ARGM/2025/000240/MTR. The third author thanks R. V. Gurjar for pointing out a mistake in an earlier version of the article. The third author thanks Aravind Asok for his encouragement and discussions. The third author also thanks Charanya Ravi, Peter Haine, Haoyang Liu for helpful discussions. A part of this work was done, while the third author was at Tata Institute of Fundamental Research, Mumbai. The third author would like to thank TIFR and the University of Southern California for providing resources during this work.

	\section{Unstable Motivic homotopy Category}
In this section, we first recall the construction of unstable $\mathbb{A}^1$-homotopy category $\mathcal{H}(k)$ from \cite{mv}, introduced by Morel-Voevodsky and we will discuss a few important properties of $\mathcal{H}(k)$. 
We prove Theorem \ref{sing-etale-bundle}, which says that an \'etale $\A^1$-bundle $\rho: X \to \mathcal{C}$ induces the Nisnevich local weak equivalence $\rho_*: Sing_*(X) \to Sing_*(\mathcal{C})$ (see also Theorem \ref{general bundle}, for similar result if $\rho$ is an ). Theorem \ref{sing-etale-bundle} is an important ingredient to prove the $\A^1$-locality of $Sing_*(X)$, if $X$ is a Koras-Russell threefold of the first kind (Theorem \ref{application purity koras}) and the $\A^1$-invariance of $\pi_0^{\A^1}(X)$, if $X$ is an $\A^1$-ruled surface (Theorem \ref{connected-base-new}).
We will also recall the fact that for every $k$-scheme $U$, $Sing_*(\mathbb{A}^n_k)(U)$ is a Kan fibrant simplicial set, by computing the horn filling formula (see Appendix \ref{full hornfill An}). \par

		\subsection{Construction and properties}
 Let $k$ be a field and $Sm/k$ be the category of finite type, smooth schemes over $k$. Let $\Delta^{op}PSh(Sm/k)$ be the category of simplicial presheaves on $Sm/k$, which we call the category of spaces over $k$. The category $\Delta^{op}PSh(Sm/k)$ has a model structure, called the projective model structure; where fibrations are sectionwise Kan fibrations, 
weak equivalences are sectionwise weak equivalences between simplicial sets and the cofibrations are defined using the left lifting property with respect to the trivial fibrations.\par
The category $Sm/k$ is equipped with the Nisnevich topology, in which the stalks are the sections over the Henselian  local schemes, which are essentially smooth schemes over $k$. A morphism $\sX \to \sY$ in $\Delta^{op}PSh(Sm/k)$ is called a Nisnevich local weak equivalence, if for every essentially smooth Henselian local scheme $\sO$ over $k$, the morphism
$$\sX(\sO) \to \sY (\sO)$$ 
is a weak equivalence of simplicial sets. Inverting the local weak equivalences, we get the Nisnevich local model structure on $\Delta^{op}PSh(Sm/k)$. The associated homotopy category is called the simplicial homotopy category, denoted by $\sH_s(k)$. Then taking the left Bousfield localisation of the Nisnevich local model structure on $\Delta^{op}PSh(Sm/k)$ with respect to the projection maps $\A^1_k\times_k X \to X$, for $X \in Sm/k$, we get the $\mathbb{A}^1$-model structure and the associated homotopy category is called the $\A^1$-homotopy category, denoted by $\sH(k)$. If we start with the category of pointed spaces over $k$, the same constructions would give the pointed simplicial homotopy category $\sH_{s, \bullet}(k)$ and the pointed $\mathbb{A}^1$-homotopy category $\mathcal{H}_\bullet(k)$ respectively.
			\begin{remark}
    			\begin{enumerate}
        \item A morphism of spaces $f: \sX \to \sY$ is called a Nisnevich local weak equivalence (or simplicial weak equivalence), if it an isomorphism in $\mathcal{H}_s(k)$ and a morphism of spaces which induces an isomorphism in $\mathcal{H}(k)$, is called an $\mathbb{A}^1$-weak equivalence.
        \item For $X \in Sm/k$ and a Nisnevich covering $f: U \to X$ in $Sm/k$, the \v{C}ech hypercover $\text{\v{C}}(f) \to X$ associated to $f$ is a Nisnevich local weak equivalence \cite[Lemma 1.15, Section 2.1]{mv}. The Nisnevich local model structure is the left Bousfield localisation of the projective model structure at the \v{C}ech hypercovers \cite[Proposition 8.1]{dugger}.
        \item For a space $\mathcal{X}$ over $k$, the projection map $\mathcal{X} \times_k \mathbb{A}^1_k \to \mathcal{X}$ is an $\mathbb{A}^1$-weak equivalence.
        \item If $\rho: X \to Y$ is an  morphism in $Sm/k$, then it is an $\mathbb{A}^1$-weak equivalence. More generally, if $\rho: X \to Y$ is an $\mathbb{A}^n$-bundle in Nisnevich topology (Definition \ref{bundle}), then it is an $\mathbb{A}^1$-weak equivalence \cite[Example 2.3, Section 3.1]{mv}. In Theorem \ref{sing-etale-bundle} and Theorem \ref{general bundle}, we will prove that if $\rho: X \to Y$ is an \'etale $\mathbb{A}^1$-bundle or an , then the induced morphism $\rho_*: Sing_*(X) \to Sing_*(Y)$ (Definition \ref{sing functor}) is a Nisnevich local weak equivalence.
    			\end{enumerate}
			\end{remark}
\begin{definition}($\A^1$-local) \cite[Definition 2.1, Section 3.2]{mv} \label{a1 local definition} \hfill
 A simplicial presheaf $\sX$ 
		is called $\A^1$-local if for any simplicial presheaf $\sY$, 
		the map
		$$Hom_{\sH_s(k)}(\sY,\sX)\to Hom_{\sH_s(k)}(\sY\times\A^1,\sX),$$
		induced by projection $\sY\times\A^1 \to \sY$, is a bijection of sets. 
\end{definition}
\begin{remark}
A space $\sX$ over $k$ is fibrant in the $\A^1$-model structure if and only if it is fibrant in the Nisnevich local model structure and it is $\mathbb{A}^1$–local
\cite[Proposition 3.19, Section 2.3]{mv}.
\end{remark}

        
       

\begin{definition} \cite[Section 2.3]{mv} \label{sing functor}
Suppose, $\mathcal{X}$ is a space over $k$.
The functor 
$$Sing_*: \Delta^{op}PSh(Sm/k) \to \Delta^{op}PSh(Sm/k)$$ 
is defined as
$$Sing_*(\mathcal{X})(U)_n=\mathcal{X}(\Delta^n_a \times_k U)_n,$$
where $\Delta^n_a$ is the cosimplicial object in $Sm/k$ defined as, 
$$\Delta^n_a = \Spec\ \frac{k[x_0, x_1,\dots, x_n]}{(\sum_{i=0}^n x_i - 1)}.$$ 
So, $\Delta^n_a \cong \mathbb{A}^n_k$. The boundary maps $d_i$ and the degeneracy maps $s_i$ are described in the Appendix \ref{maps of sing}.
\end{definition}

\begin{remark}
\begin{enumerate}
\item The canonical morphism $\mathcal{X} \to Sing_*(\mathcal{X})$ is an $\mathbb{A}^1$-weak equivalence \cite[Corollary 3.8, Section 2.3]{mv}.
\item If $f: \mathcal{X} \to \mathcal{Y}$ is 
a fibration in the $\A^1$-model structure, then the morphism $Sing_*(f): Sing_*(\mathcal{X}) \to Sing_*(\mathcal{Y})$ is also a fibration in the $\A^1$-model structure
\cite[Corollary 3.13, Section 2.3]{mv}.
\end{enumerate}
\end{remark}

\begin{definition} \label{kan fibrant definition} \cite[Definition 1.3]{may}
A simplicial set $X$ is called Kan fibrant if for every $n$ and 
given $l$-th horn in $X$ i.e. a map $\phi: \Lambda^n_l \to X$, $0 \leq l \leq n$, there is a map $\Phi: \Delta^n \to X$ such that $\phi = \Phi \circ \theta$, where $\theta: \Lambda^n_l \to \Delta^n$ is the inclusion of the $l$-th horn in $\Delta^n$. Equivalently, 
given $n$-many $(n-1)$ simplices $x_0,\dots, x_{l-1}, x_{l+1},\dots, x_n$ of $X$ satisfying the compatibility condition 
$$d_i x_j = d_{j-1} x_i, \text{ for }i < j, i, j \neq l, \ d_i: X_n \to X_{n-1} \text{ are the boundary maps },$$
then there is an $n$-simplex $x$ such that $d_i x  = x_i$ for all $i \neq l$. 
\end{definition}
\begin{example} \label{kan fibrant An}
Any simplicial group is Kan fibrant \cite[Theorem 17.1]{may}. Thus for any $k$-scheme $U$, $Sing_*(\mathbb{A}^m_k)(U)$ is Kan fibrant, since 
$\mathbb{A}^m_k$ is an affine group scheme (see also Remark \ref{boundary and horn filling}(3)). The retract of a Kan fibrant simplicial set is Kan fibrant.
\end{example}
Using \cite[Lemma 3.1]{curtis-simplicial} we give an explicit formula of $2$-horn filling of the sections of $Sing_*(\mathbb{A}^m_k)$ in the next example, over $U \in Sm/k$.

\begin{example}{[$2$-horn filling formula for $Sing_*(\A^m_k)(U)$, see also Section \ref{full hornfill An}]} \label{2 horn}
	Let $U$ be any $k$-scheme and
	suppose, we are given a morphism 
	$$\phi: \Lambda^2_l \to Sing_*(\mathbb{A}^m_k)(U), \text{ where } U \in Sm/k, l \in \{0, 1, 2\}.$$
	We will extend $\phi$ to a $2$-simplex of $Sing_*(\mathbb{A}^m_k)(U)$.
	We have the following isomorphism:
	\begin{align*}
		Sing_*(\mathbb{A}^m_k)(U)_n & = Hom_{Sm/k}(\Delta^n_a \times_k U, \mathbb{A}^m_k) \\
		& \cong Hom_{k-alg}(k[T_1,\dots, T_m], \frac{A[x_0,\dots, x_n]}{(\sum_{i=0}^n x_i - 1)}),
	\end{align*}
	where $A$ is the ring of regular functions $\mathcal{O}(U)$.
	Here are three cases according to the $l$-th horn, $l = 0,1,2$. \par
	\begin{enumerate}[start=1,label={\bfseries Case \arabic*:}]
		\item $l=0$.

 Suppose we are given
		$$\phi_1, \phi_2: k[T_1,\dots,T_m] \to \frac{A[x_0, x_1]}{(x_0 + x_1 - 1)}$$
		such that $d^1 \circ \phi_2 = d^1 \circ \phi_1$.  Since $T_j$'s are the independent variables, suppose $\phi_1, \phi_2$ are given by
		$$T_j \xmapsto{\phi_1} \overline{f_j}, \ f_j \in A[x_0, x_1] \text{ and}$$
		$$T_j \xmapsto{\phi_2} \overline{g_j}, \ g_j \in A[x_0, x_1]$$
		and $f_j, g_j$ satisfy
		$$\overline{g_j(x_0, 0)} = \overline{f_j(x_0, 0)}.$$
		Define $\Phi: k[T_1,\dots,T_m] \to \frac{A[x_0, x_1, x_2]}{(x_0 + x_1 +x_2-1)}$
		as 
		$$T_j \mapsto \overline{g_j(x_0, x_1+x_2) + f_j(x_0+x_1, x_2) - g_j(x_0 + x_1 , x_2) }.$$
		Then $d^1 \circ \Phi$ is given by 
		$$T_j \mapsto \overline{g_j(x_0, x_1) - g_j(x_0, x_1)+ f_j(x_0, x_1)}$$
		so, $T_j \xmapsto{d^1 \circ \Phi} \overline{f_j(x_0, x_1)}$. Thus, $d^1 \circ \Phi = \phi_1$. \\
		The map $d^2 \circ \Phi$ is given by
		$$T_j \mapsto \overline{g_j(x_0, x_1) - g_j(x_0+x_1, 0) + f_j(x_0+x_1, 0)}.$$
		Since $\overline{g_j(x_0, 0)} = \overline{f_j(x_0, 0)} \in \frac{A[x_0]}{(x_0 -1)}$, so $T_j \xmapsto{d^2 \circ \Phi} \overline{g_j(x_0, x_1)}$. Thus, $d^2 \circ \Phi = \phi_2$.
		\item $l = 1$.

		Suppose, we are given
		$$\phi_0, \phi_2: k[T_1,\dots,T_m] \to \frac{A[x_0, x_1]}{(x_0 + x_1 -1)}$$
		such that $d^0 \circ \phi_2 = d^1 \circ \phi_0$. Suppose $\phi_0, \phi_2$ are given by
		$$T_j \xmapsto{\phi_0} \overline{f_j}, \ f_j \in A[x_0, x_1] \text{ and}$$
		$$T_j \xmapsto{\phi_2} \overline{g_j}, \ g_j \in A[x_0, x_1]$$
		and $f_j, g_j$ satisfy
		$$\overline{g_j(0, x_0)} = \overline{f_j(x_0, 0)}.$$
		Define $\Phi: k[T_1,\dots,T_m] \to \frac{A[x_0, x_1, x_2]}{(x_0 + x_1 +x_2-1)}$
		as 
		$$T_j \mapsto \overline{f_j(x_0+ x_1,x_2) + g_j(x_0, x_1+ x_2) - f_j(x_0 + x_1+ x_2, 0) }, $$
		$$\text{so } T_j \xmapsto{\Phi} \overline{f_j(x_0+x_1, x_2) + g_j(x_0, x_1 + x_2) - f_j(1, 0) }.$$
		Similarly as in Case 1, we can check $d^0 \circ \Phi = \phi_0$ and $d^2 \circ \Phi = \phi_2$. 
		\item $l = 2$.

		Suppose, we are given
		$$\phi_0, \phi_1: k[T_1,\dots,T_m] \to \frac{A[x_0, x_1]}{(x_0 + x_1 -1)}$$
		such that $d^0 \circ \phi_1 = d^0 \circ \phi_0$. Suppose $\phi_0, \phi_1$ are given by
		$$T_j \xmapsto{\phi_0} \overline{f_j}, \ f_j \in A[x_0, x_1] \text{ and}$$
		$$T_j \xmapsto{\phi_1} \overline{g_j}, \ g_j \in A[x_0, x_1]$$
		and $f_j, g_j$ satisfy
		$$\overline{g_j(0, x_0)} = \overline{f_j(0, x_0)}.$$
		Define $\Phi: k[T_1,\dots,T_m] \to \frac{A[x_0, x_1, x_2]}{(x_0 + x_1 +x_2-1)}$
		as 
		$$T_j \mapsto \overline{f_j(x_0+ x_1,x_2) + g_j(x_0, x_1+ x_2) - f_j(x_0, x_1+ x_2) }.$$
		$$T_j \xmapsto{\Phi} \overline{f_j(x_0+x_1, x_2) - f_j(1, 0) + g_j(x_0, x_1 + x_2)}.$$
		Similarly as in Case 1, we can check $d^0 \circ \Phi = \phi_0$ and $d^1 \circ \Phi = \phi_1$. 
		
		Therefore, for any horn $\phi: \Lambda^2_l \to Sing_*(\A^m_k)(U)$, where $0 \leq l \leq 2$, there is a $2$-simplex $\Phi:\Delta^2 \to Sing_*(\A^m_k)(U)$ that extends the horn $\phi$, that is, $\phi= \Phi \circ \theta$, where $\theta: \Lambda^2_l \to \Delta^2$ is the inclusion (Definition \ref{kan fibrant definition})
	\end{enumerate}
\end{example}
In \ref{full hornfill An} using \cite[Lemma 3.1]{curtis-simplicial}, we write $l$-th horn (inside $\Delta^n$) filling formula, for any $l,n$ with $0 \leq l \leq n$.




\begin{remark} \cite[Example 2.19]{choudhury2}
Thus in particular, if we are given two morphisms $f, g : \mathbb{A}^1_k \to \mathbb{A}^m_k$ such that $f(1) = g(0)$, then the morphism
$h: \mathbb{A}^2_k \to \mathbb{A}^m_k$ defined as
$$(t_1, t_2) \mapsto (f(1-t_1)+g(t_2) - f(1))$$
satisfies $h(1-x, 0) = f(x)$ and $h(0, x) = g(x)$. 
Therefore, if there are two intersecting $\mathbb{A}^1$'s in $\mathbb{A}^2_k$, then we can extend it to get a morphism from $\mathbb{A}^2_k$ to $\mathbb{A}^2_k$.
\end{remark}
We end this subsection with the following lemma about section-wise simplicial contractibility of $Sing_*(\A^m_k)$ (see also Remark \ref{boundary and horn filling}(5)). 


\begin{lemma} \cite[Section 2.3, Corollary 3.5]{mv} \label{contractible An}
    Suppose, $U\in Sm/k$. Then the Kan fibrant simplicial set $Sing_*(\mathbb{A}^n_k)(U)$ is contractible.
\end{lemma}
\begin{proof}
    The multiplication map 
    $$\gamma: \mathbb{A}^n_k \times_k \mathbb{A}^1_k \to \mathbb{A}^n_k \text{ given by } (s,t) \mapsto st$$
    gives naive $\mathbb{A}^1$-homotopy between the identity map and the constant map to a $k$-point. Since $Sing_*$ functor commutes with the product, so $\gamma$ induces a map
    $$Sing_*(\gamma): Sing_*(\mathbb{A}^n_k) \times Sing_*(\mathbb{A}^1_k) \to Sing_*(\mathbb{A}^n_k).$$
    Consider the $1$-simplex in $\theta= \mathrm{Id}_{\A^1}\in Sing_*(\A^1_k)$, given by $\theta(0)=\Spec(k)\xrightarrow{s_0}\A^1_k,\mathrm{\ and \ } \theta(1)=\Spec(k)\xrightarrow{s_1}\A^1_k$.
    Consider the composition of morphisms
    $$\Gamma: Sing_*(\mathbb{A}^n_k) \times \Delta^1 \xrightarrow{(Id, \theta)} Sing_*(\mathbb{A}^n_k) \times Sing_*(\mathbb{A}^1_k) \xrightarrow{Sing_*(\gamma)} Sing_*(\mathbb{A}^n_k).$$ 
    Then we have -
    \begin{align*}
    	\Gamma(0) &= Sing_*(\A^n_k)\xrightarrow{\mathrm{Const}_{(0,\dots,0)}} Sing_*(\A^n_k)\\
    	\Gamma(1) &= Sing_*(\A^n_k) \xrightarrow{\mathrm{Id}} Sing_*(\A^n_k).
    \end{align*}
     This defines a simplicial homotopy between identity map and a constant map. Hence $Sing_*(\A^n_k)(U)$ is contractible for any $U\in Sm/k$.
\end{proof}

\begin{remark} (see also Remark \ref{boundary and horn filling}(5)) \label{boundary lifting question}
     Let $U$ be a smooth $k$-scheme. Then $Sing_*(\mathbb{A}^m_k)(U)$ is Kan fibrant simplicial set, which is also contractible (Example \ref{kan fibrant An} and Lemma \ref{contractible An}). Therefore, the constant map $Sing_*(\mathbb{A}^m_k)(U) \to \bullet$ is a trivial Kan fibration. Therefore, $Sing_*(\mathbb{A}^m_k)(U)$ has boundary lifting property:
    
		\[\xymatrix{\partial \Delta^n \ar@{^{(}->}[d]^i \ar[r] &Sing_*(\mathbb{A}^m_k)(U)\\
		\Delta^n \ar@{.>}[ru]^{\exists}} \]
        This means, given a collection of morphisms $f_0, \dots, f_n: \Delta^{n-1}_a \times U \to \mathbb{A}^m_k$ satisfying the compatibility conditions (i.e., $d^if_j = d^{j-1}f_i$ for $0 \leq i<j \leq n$), there exists an $n$-simplex $f: \Delta^n_a \times U \to \mathbb{A}^m_k$ such that $d_i(f) = f_i$, for every $i$. 
        
\end{remark}

\subsection{A local weak equivalence}
	In this subsection we prove Theorem \ref{sing-etale-bundle},
which will be used substantially in the later sections. We start by recalling few definitions.  

 We recall the definition and few properties of $\mathbb{A}^1$-fibration from \cite[Chapter 2]{masuda} and \cite[Chapter 3]{miyanishicrm}. Let 
$k$ be an algebraically closed field of characteristic zero.
	\begin{definition} \label{bundle}
 Let $\mathcal{C}$ be a smooth algebraic space over $k$ and $X$ be a smooth variety over $k$. A morphism $\pi: X \to \sC$ is called an $\A^1$-bundle (with respect to the Zariski, Nisnevich or \'etale topology),
        if there exists a covering (Zariski, Nisnevich or \'etale covering respectively)
        $U \to \sC$ such that the fiber product $X \times_\sC U$, 
        which is a scheme 
       \cite[Corollary 4.2.3]{alpher},
        is isomorphic to $\A^1_k \times_k U$ over $U$. If $\pi$ is an $\mathbb{A}^1$-bundle with respect to the Zariski topology, then we will simply call $\pi$ to be an $\mathbb{A}^1$-bundle.
	\end{definition}
	\begin{remark} \label{a1 bundle properties}
	    \begin{enumerate}
	    
          \item Every fiber over a point $x$ of an $\mathbb{A}^1$-bundle is isomorphic to $\mathbb{A}^1_{\kappa(x)}$, where $\kappa(x)$ is the residue field of $x$. 
\item An $\A^1$-bundle $\pi: X \to \sC$ in the Nisnevich topology (hence in the Zariski topology also) is an $\A^1$-weak equivalence \cite[Section 3.1, Example 2.3]{mv}. In Theorem \ref{sing-etale-bundle}, we will prove that if $\pi$ is an \'etale $\A^1$-bundle, then $\pi_*: Sing_*(X) \to Sing_*(\sC)$ is a Nisnevich local weak equivalence. Therefore, an $\A^1$-bundle (in any topology Zariski, Nisnevich or \'etale) is an $\A^1$-weak equivalence (Corollary \ref{factorisation A1 equivalence}).
	        \item A line bundle is an $\mathbb{A}^1$-bundle (with respect to the Zariski topology). But the converse is not true. The $\A^1$-bundle in \cite[Example 2]{doub} does not admit a section, therefore it is not a line bundle.
            \item An $\mathbb{A}^1$-bundle corresponds to a line bundle, if the base is an affine scheme. Indeed, a line bundle over a scheme $X$ correspond to a $\mathbb{G}_m$-torsor over $X$; which correspond to an element of $H^1_{\text{Zar}}(X, \mathbb{G}_m)$ \cite[Corollary 4.7]{milne}. On the other hand, an $\mathbb{A}^1$-bundle over $X$ corresponds to an $\Aut(\A^1_k)$-torsor over $X$ (where, $\Aut(\A^1_k)$ is the automorphism group of the affine line $\A^1_k$); which corresponds to an element of $H^1_{\text{Zar}}(X, \Aut(\A^1_k))$. But the Zariski group sheaf $\Aut(\A^1_k)$ is isomorphic to the semidirect product
            $\mathbb{G}_m \ltimes \mathbb{G}_a$, which gives the short exact sequence of Zariski sheaves 
		$$0 \rightarrow \mathbb{G}_a \rightarrow \Aut(\A^1_k) \rightarrow \mathbb{G}_m \rightarrow 0.$$
		This gives a long exact sequence in cohomology of $X$ (in Zariski topology)
		$$\dots \rightarrow H^1_{\text{Zar}}(X, \mathbb{G}_a) \rightarrow H^1_{\text{Zar}}(X, \Aut(\A^1_k)) \rightarrow H^1_{\text{Zar}}(X, \mathbb{G}_m) \rightarrow 
H^2_{\text{Zar}}(X, \mathbb{G}_a) \rightarrow 
\dots$$
        Since $X$ is affine, so 
        $H^1_{\text{Zar}}(X, \mathbb{G}_a) \cong H^2_{\text{Zar}}(X, \mathbb{G}_a) \cong \{0\}$. Therefore,
        $$H^1_{\text{Zar}}(X, \Aut(\A^1_k)) \xrightarrow{\cong} H^1_{\text{Zar}}(X, \mathbb{G}_m).$$
	    \end{enumerate}
	\end{remark}
\begin{definition} \cite[Definition 2.1.1]{masuda} \label{a1 fibration definition}
Suppose, $F$ is a variety over $k$. A surjective morphism of smooth varieties 
        $$\pi: X \to C$$ 
   over $k$ is called an $F$-fibration, if the generic fiber $X_\eta := X\times_C \Spec \kappa(\eta)$ is isomorphic to $F \times_k \Spec \kappa(\eta)$ over $\kappa(\eta)$,
where $\eta$ is the generic point of $C$. If $F$ is isomorphic to $\A^1_k$, then we call the morphism $\pi$ to be an $\A^1$-fibration. 
	\end{definition}
\begin{notation}
For a morphism $\pi: X \to C$ of smooth varieties over $k$ and a point $a \in C$ with the residue field $\kappa(a)$, the fiber $\pi^{-1}(a)$ denotes the scheme-theoretic fiber of $\pi$ over $a$, which is by definition the fiber product $X \times_{\pi, a} \Spec (\kappa(a))$.
\end{notation}	
	\begin{remark} \label{properties a1 fibration}
Suppose, $X$ is a smooth affine surface and the base of the $\mathbb{A}^1$-fibration $\pi: X \to C$ is a smooth curve $C$. The base curve $C$ consists of the closed points (which have residue fields $k$) and the generic point $\eta$. 
    \begin{enumerate}
        \item The generic fiber of $\pi$ is isomorphic to $\mathbb{A}^1_K$, where $K=\kappa(\eta)$ is the function field of $C$ over $k$ \cite[Section 2.2]{kambayashi}.
        \item A closed fiber of $\pi$ is isomorphic to $\mathbb{A}^1_k$ if it is reduced and irreducible. 
        \item A closed fiber $F=\pi^{-1}(x)$ is called degenerate, if it is either reducible or non-reduced \cite[Definition 4]{doub}. Suppose,
        $F = \cup_{i} F_i$; $F_i$'s are the connected components of $F$. Then the local ring $\mathcal{O}_{F_i, \xi_i}$ of $F_i$ at its generic point $\xi_i$ is an Artinian local ring of length $m_i$ (say). The length $m_i =1$ if and only if $F_i$ is reduced. Scheme-theoretically, the fiber $F$ can be expressed as $$F=\sum_{i=0}^n m_i F_i.$$
        The associated reduced structure $(F_i)_{\text{red}}$ of $F_i$ is isomorphic to $\mathbb{A}^1_k$
        \cite[Chapter 3, Lemma 1.4.2]{miyanishicrm}.
   \item An $\mathbb{A}^1$-fibration is generically trivial i.e. there is an open subset $U$ of $C$ such that the fiber $X \times_C U$ over $U$ is isomorphic to $\mathbb{A}^1_U$ \cite[Lemma 1.1]{threefolds}. So in particular, there are only finitely many points $x \in C(k)$ such that $\pi^{-1}(x)$ is not isomorphic to $\mathbb{A}^1_k$. 
   \end{enumerate}
	\end{remark}
	
	\begin{example} [\textbf{$\mathbb{A}^1$-fibrations}]\label{d-sphere}\hfill
    \begin{enumerate}
        \item The first family of examples of $\mathbb{A}^1$-fibrations are $\mathbb{A}^1$-bundles, since every fiber over a point of an $\mathbb{A}^1$-bundle
         is isomorphic to $\mathbb{A}^1$ (Remark \ref{a1 bundle properties}).
The converse is true for an $\A^1$-fibration $\pi: X \to C$ from a smooth affine surface $X$ to a smooth curve $C$, if every scheme-theoretic fiber of $\pi$ is irreducible and reduced \cite[Proposition 3]{doub}.
        \item (\textbf{Danielewski surface} \cite[Chapter 5, Section 5.3.4]{as}) An $\A^1$-fibration is not necessarily an $\A^1$-bundle or an $\A^1$-weak equivalence. 
For example, suppose $X$ is a Danielewski surface given by
$$X \equiv \{x^nz=P(y)\} \subset \mathbb{A}^3_k = \Spec\ k[x,y,z],$$ 
where $P(y) \in k[t]$ is a separable polynomial.
If the degree of $P(y)$ is at least $2$ then $k[x,y,z]/(x^nz-P(y))$ is not a UFD, so the surface $X$ is not isomorphic to affine plane. Indeed, if $deg(P(y))=d \geq 2$, then we can write $P(y) = a \prod_{i=1}^{d}(y - \alpha_i)$, for some $a, \alpha_1, \dots, \alpha_d \in k$. So in $\sO(X)$, there are two different ways to express $x^nz$ as products of irreducibles
$$x^nz = a \prod_{i=1}^{d}(y - \alpha_i) \in \sO(X).$$
Thus $\sO(X)$ is not a UFD.

Consider the morphism 
$$\text{pr}_x: X \to \mathbb{A}^1_k \text{ defined as } (x,y,z) \mapsto x$$
is an $\A^1$-fibration, since 
$$\text{pr}_x^{-1}(\mathbb{G}_m = \{x \neq 0\}) \cong \mathbb{A}^1_k \times_k \mathbb{G}_m.$$ 
In this example, all fibers of $\text{pr}_x$ are reduced and only the fiber over $0$ is reducible, which is the disjoint union of $d$-copies of $\mathbb{A}^1_k$'s, where $d=deg(P(y))$. The surface $X$ is $\A^1$-weak equivalent to the wedge sum of $(d-1)$-copies of $\mathbb{P}^1_k$ \cite[Chapter 5, Proposition 5.3.4.3]{as} (which is again $\A^1$-weakly equivalent to the affine line with $d$-many origins, see Section \ref{section a1 homotopy types}). 
In particular, if we take $d=2$ and $X \equiv \{x^2z=y(y-1)\}$, then $X$ is $\A^1$-weakly equivalent to $\mathbb{P}^1_k$ (which is again $\A^1$-weakly equivalent to the affine line with double origin, see Section \ref{section a1 homotopy types}).


        \item An $\mathbb{A}^1$-fibration may have a non-reduced fiber. Consider the surface, $X \subseteq \Spec\ \mathbb{C}[x,y,z]$ given by $x^2z = y^2-x$, which admits an $\mathbb{A}^1$-fibration \cite[Example 5]{doub}
        $$\text{pr}_x: X \to \mathbb{A}^1_{\mathbb{C}} \text{ given by } (x,y,z) \mapsto x..$$
        Here the fiber $\text{pr}_x^{-1}(a)$ is isomorphic to $\mathbb{A}^1_\mathbb{C}$, if $a$ is a non-zero closed point of $\mathbb{A}^1_{\mathbb{C}}$. If $a=0$, the fiber $\text{pr}_x^{-1}(0)$ is isomorphic to the scheme $\Spec\ (\frac{\mathbb{C}[y]}{(y^2)}[z])$, therefore it is non-reduced, but the associated reduced structure $\text{pr}^{-1}_x(0)_{\text{red}}$ is isomorphic to $\A^1_k$ (see also Example \ref{multi-section example}).
    \end{enumerate}
	\end{example}
\begin{remark}
The $\A^1$-fibrations are a particular class of morphisms, used to study the affine surfaces. An $\A^1$-fibration is not a fibration in the $\A^1$-model structure, in general. For example, suppose $X$ is the affine surface as in Example \ref{d-sphere}(3). The morphism $\text{pr}_x$ is not a fibration in the $\A^1$-model structure. Indeed, $\text{pr}_x$ does not admit a section such that the following diagram has a lift;
\[\xymatrixcolsep{3pc}
\xymatrix{\Spec\ k \ar[d]_{0} \ar[r]^{(0,0,0)} & X \ar[d]^{\text{pr}_x}\\
			\A^1_k\ar@{.>}[ru]^{\nexists} \ar@{=}[r]&\A^1_k} 
\]
that is, there does not exist any morphism $\A^1_k \to X$ such that both the triangles commute (see Example \ref{multi-section example}(3)). So $\text{pr}_x$ does not have right lifting property with respect to the trivial $\mathbb{A}^1$- cofibration $\Spec k \xrightarrow{s_0} \A^1_k$.
\end{remark}
	Let us recall the structure theorem for $\A^1$-fibration:
	\begin{theorem} \cite[Theorem 9																																																																																																																																]{doub}\label{factor-A1-fibration}
		Suppose $X$ is a smooth affine surface over $k$ and $\pi: X \to C$ is an $\mathbb{A}^1$-fibration over a smooth curve $C$. Then $\pi$ factors as an \'etale locally trivial $\mathbb{A}^1$-bundle $\theta: X \to \mathcal{C}$ over a smooth algebraic space $\mathcal{C}$ of dimension $1$ (possibly not a scheme) along with the structure morphism $\alpha: \mathcal{C} \to C$, which is surjective, quasi-finite and birational morphism of finite type.
	\end{theorem}
   
    \begin{construction} \label{construction of algebraic space}
Following \cite[Theorem 9]{doub}, the construction of the algebraic space $\mathcal{C}$ in Theorem \ref{factor-A1-fibration} consists of two steps: 
\begin{enumerate}
	\item 
	Suppose, $c_1, \dots, c_r$ are finitely many $k$-points of $C$ such that $\pi^{-1}(c_i)$ is not irreducible and $\pi^{-1}(c_i)$ consists $d_i$-many irreducible components.
	The morphism $\pi: X \to C$ is factored as 
	$$X \xrightarrow{\check{\pi}} \check{C} \xrightarrow{\check{p}} C,$$
	where $\check{C}$ is a non-separated scheme obtained by replacing the points $c_i$ by $d_i$-many points and $\check{p}: \check{C} \to C$ is the canonical morphism (see also \cite[Example 7]{doub}). Every fiber of $\check{\pi}$ is irreducible 
	(fiber of $\check{\pi}$ is non-reduced if and only if the corresponding fiber of $\pi$ is non-reduced) and thus the reduced structure of each fiber is isomorphic to $\A^1$. If every fiber of $\pi$ is irreducible, then $\check{C}$ is just $C$ itself and $\check{p}$ is the identity map.
	\item The algebraic space $\sC$ is obtained in finitely many steps, one for each non-reduced fiber of $\check{\pi}$. At each step, for a non-reduced fiber of $\check{\pi}$ at $\check{c} \in \check{C}$, an intermediate algebraic space is constructed by replacing a Zariski neighbourhood of $\check{c}$ in $\check{C}$, using an \'etale equivalence relation. The morphism $\check{\pi}$ is further factored as 
	$$X \xrightarrow{\theta} \mathcal{C} \xrightarrow{p} \check{C},$$
	where $\theta$ is an \'etale locally trivial $\A^1$-bundle and $\alpha = \check{p} \circ p$. As in (a), if every fiber of $\pi$ is reduced, then the algebraic space $\sC$ is just $\check{C}$ itself (a scheme) and $p$ is the identity map. 
	\item A closed point of $\sC$ corresponds to a closed point of $\check{C}$ and a closed point $\check{c} \in \check{C}$ corresponds to an irreducible component (which is also a connected component) of the fiber $\pi^{-1}(\check{p}(\check{c}))$.
\end{enumerate}
    \end{construction}

          \begin{remark} \label{comments of structure theorem}
          	If every fiber $\pi^{-1}(x)$ is isomorphic to $\mathbb{A}^1$, then $\pi$ is actually an $\mathbb{A}^1$-bundle \cite[Proposition 3]{doub}. An $\mathbb{A}^1$-bundle between schemes is an $\mathbb{A}^1$-weak equivalence \cite[Theorem 5.3.1.3]{as}.
         
    \end{remark}

Now we will prove the main theorem in this article (Theorem \ref{sing-etale-bundle}), which implies that the morphism $\theta$ in Theorem \ref{factor-A1-fibration} induces
a Nisnevich local weak equivalence $\theta_*: Sing_*(X) \to Sing_*(\mathcal{C})$. So, $\theta$ is an $\A^1$-weak equivalence.
\begin{theorem} \label{sing-etale-bundle}
		Let $X$ be a smooth variety over $k$ and $\mathcal{C}$ be a smooth algebraic space over $k$. Suppose, $\rho: X \to \mathcal{C}$ is an \'etale locally trivial $\mathbb{A}^1$-bundle.
		Then the induced morphism $\rho_*: Sing_*(X) \to Sing_*(\mathcal{C})$ is a Nisnevich local weak equivalence.
	\end{theorem}
	\begin{proof}
		We will show that for every essentially smooth Henselian local scheme $\mathcal{O}$ over $k$, the morphism between the simplicial sets
		$$\rho_*: Sing_*(X)(\mathcal{O}) \to Sing_*(\mathcal{C})(\mathcal{O})$$
		is a trivial Kan fibration.
		Therefore, we need to find a lift of the following given commutative square

        \begin{equation}
		\label{lift}
	\xymatrix{\partial \Delta^n \ar@{^{(}->}[d]^i \ar[r]^{f \ \ \ \ } &Sing_*(X)(\sO)\ar[d]^{\rho_*}\\
	\Delta^n \ar@{.>}[ru]^{\exists} \ar[r]^-{g} &Sing_*(\sC)(\sO)} 
       \end{equation}  
The $n$-simplex $g: \Delta^n \to Sing_*(\sC)(\sO)$ is given by a morphism (we denote it by $g$ again) $g: \A^n_\sO \to \sC$.	 The morphism $f: \partial \Delta^n \to Sing_*(X)(\mathcal{O})$ is a collection of $(n+1)$-many morphisms $f_j:\mathbb{A}^{n-1}_{\mathcal{O}} \to X$ ($0 \leq j \leq n$) satisfying the compatibility conditions ($d^if_j = d^{j-1}f_i$ for $0 \leq i<j \leq n$). This is 
equivalent to find the lift of the diagram (see Remark \ref{boundary and horn filling}(5))
 \begin{equation}
  \label{liftequivalent}
\xymatrix{
\partial\Delta^n_\sO \ar[r]^{f} \ar[d]_i & X  \ar[d]^\rho\\
\mathbb{A}^n_{\mathcal{O}} \ar[r]^{g}\ar@{..>}[ur]^{\exists} & \mathcal{C}
}
\end{equation}
where $\partial\Delta^n_\sO:= \partial \Delta^n_a \times_k \sO$ (Remark \ref{boundary and horn filling}(4)).
where $d^p_q$ is the boundary map $d^p:\A^{n-1}_\sO \to \A^n_\sO$ to the $q$-th component in the coproduct (see Section \ref{detailed horn formula}, for the boundary maps $d^i$).
		Consider the pullback diagram in the category of algebraic spaces

        \begin{equation}
        	\label{pullback}
        \xymatrix{\mathbb{A}^n_\mathcal{O}\times_\mathcal{C}X \ar[d]^{\pr_1} \ar[r]^{\ \ \ \pr_2} &X \ar[d]^{\rho}\\
        	\mathbb{A}^n_\mathcal{O} \ar[r]^g & \sC }
        \end{equation}
        
		Since $\rho$ is an étale locally trivial $\mathbb{A}^1$-bundle, so $\text{pr}_1$ is also an étale locally trivial $\mathbb{A}^1$-bundle.
The pullback algebraic space $\mathbb{A}^n_{\mathcal{O}} \times_{\mathcal{C}} X$ is actually a scheme \cite[Corollary 4.2.3]{alpher}.
        We claim that 
		$\text{pr}_1$
		is isomorphic to the trivial line bundle over $\mathbb{A}^n_{\mathcal{O}}$. Indeed, an \'etale $\A^1$-bundle over $\A^n_\sO$ is given by a class in $H^1_{\text{\'et}}(\A^n_\sO,\Aut(\A^1_k)) \cong H^1_{\text{Zar}}(\A^n_\sO,\Aut(\A^1_k))$ and furthermore, a Zariski $\A^1$-bundle is isomorphic to a line bundle, if the base is affine
(Remark \ref{a1 bundle properties}(4)).
So,
$$H^1_{\text{\'et}}(\A^n_\sO,\Aut(\A^1_k)) \cong H^1_{\text{Zar}}(\A^n_\sO,\Aut(\A^1_k)) \cong H^1_{\text{Zar}}(\A^n_\sO, \G_m) \cong Pic(\mathbb{A}^n_{\mathcal{O}}) = 0,$$
		since the Picard group functor is $\mathbb{A}^1$-invariant and any projective module over a local ring is free. 
		 Hence, the \'etale $\mathbb{A}^1$-bundle $\text{pr}_1: \mathbb{A}^n_{\mathcal{O}} \times_{\mathcal{C}} X \to \mathbb{A}^n_{\mathcal{O}}$ is isomorphic to the projection $p_1: \mathbb{A}^n_{\mathcal{O}} \times_k \mathbb{A}^1_k \to \mathbb{A}^n_{\mathcal{O}}$.
		Thus we have the following commutative diagram

		\[
		\xymatrix{\mathbb{A}^n_\mathcal{O}\times_k \mathbb{A}^1_k \ar[rd]_{p_1} \ar[r]^{\cong}_\phi &\mathbb{A}^n_\mathcal{O}\times_\mathcal{C}X \ar[d]^{\pr_1} \ar[r]^<<<<{\pr_2} &X \ar[d]^{\rho}\\
		 &\A^n_{\sO} \ar[r]^g &\sC}
		\]
		Since $p_1$ is the trivial bundle, so it admits a global section and a global section of $p_1$ is determined by a morphism $\mathbb{A}^n_{\mathcal{O}} \to \mathbb{A}^1_k$. 
		The commutative diagram \eqref{liftequivalent} and the pullback diagram \eqref{pullback} induces the morphism 
		$$h:  \partial\Delta^n_\sO \to  \mathbb{A}^n_{\mathcal{O}} \times_{\mathcal{C}} X. $$
		Consider the composition of the morphisms 
		$$p_2 \circ \phi^{-1} \circ h:  \partial\Delta^n_\sO \to \mathbb{A}^1_k,$$
		where $p_2: \mathbb{A}^n_{\mathcal{O}} \times_k \mathbb{A}^1_k \to \mathbb{A}^1_k$ is the projection.
		Now for any $k$-scheme $U$, the simplicial set $Sing_*(\mathbb{A}^1_k)(U)$ is a Kan fibrant simplicial set, which is also contractible simplicial set (Example \ref{kan fibrant An} and Lemma \ref{contractible An}). So the following commutative diagram has a lift

		\[\xymatrixcolsep{5pc}
		\xymatrix{ \partial\Delta^n_\sO \ar@{^{(}->}[d]^i \ar[r]^{\ \ \ p_2\circ \phi^{-1}\circ h} &{\mathbb{A}^1_k}\\
		\A^n_\sO \ar@{.>}[ur]^{\exists s}}
		\]
Let $\Tilde{s}: \mathbb{A}^n_{\mathcal{O}} \to \mathbb{A}^n_{\mathcal{O}} \times_k \mathbb{A}^1_k$ be the section  of $p_1$, which is determined by the lift $s: \mathbb{A}^n_{\mathcal{O}} \to \mathbb{A}^1_k$. Then the desired lift of the diagram \eqref{liftequivalent} is given by 
		$$\text{pr}_2 \circ \phi \circ \Tilde{s}: \mathbb{A}^n_{\mathcal{O}} \to X.$$
		
		Indeed,
		\begin{align*}
			\rho \circ \text{pr}_2 \circ \phi \circ \Tilde{s} &= g \circ \text{pr}_1 \circ \phi \circ \Tilde{s} \\
			& = g \circ p_1 \circ \Tilde{s} \\
			& = g,
		\end{align*}
so the lower triangle in diagram \ref{liftequivalent} commutes and for the commutativity of the upper triangle in diagram \ref{liftequivalent}, we will prove that 
$$\text{pr}_2 \circ \phi \circ \Tilde{s} \circ d^i = f_i, \text{ for every }i, 0 \leq i \leq n$$
where $d^i: \A^{n-1}_\sO \to \A^n_\sO$ is the $i$-th face inclusion (see Section \ref{detailed horn formula}).
Now, for $u \in  \mathbb{A}^{n-1}_{\mathcal{O}}$,
		\begin{align*}
			\text{pr}_2 \circ \phi \circ \Tilde{s} \circ d^i (u) &=  \text{pr}_2 \circ \phi(d^i(u),(s\circ d^i)(u))\\
			&=\text{pr}_2 \circ \phi(d^i(u), p_2\circ \phi^{-1} \circ h(u))\\
			&=\text{pr}_2 \circ \phi(d^i(u), p_2\circ \phi^{-1}(d^i(u),f_i(u))) \\
			&=f_i(u).
		\end{align*}
Hence, diagram \ref{liftequivalent} has a lift.
		Therefore, 
$\rho_*: Sing_*(X)(\mathcal{O}) \to Sing_*(\mathcal{C})(\mathcal{O})$ is a trivial Kan fibration. Hence the morphism $\rho_*: Sing_*(X) \to Sing_*(\mathcal{C})$ is a Nisnevich local weak equivalence.
		
	\end{proof}
    Since the canonical morphism $\mathcal{X} \to Sing_*(\mathcal{X})$ is an $\mathbb{A}^1$-weak equivalence \cite[Corollary 3.8, Section 2.3]{mv}, therefore we have the following immediate corollary:
	\begin{corollary} \label{factorisation A1 equivalence}
	    Let $X$ be a smooth variety over $k$ and $\mathcal{C}$ be a smooth algebraic space over $k$. Suppose, $\rho: X \to \mathcal{C}$ is an \'etale locally trivial $\mathbb{A}^1$-bundle. Then $\rho$ is an $\mathbb{A}^1$-weak equivalence. So, the morphism $\theta$ which appears in the factorisation of the $\A^1$-fibration $\pi:X \to C$ in Theorem \ref{factor-A1-fibration} is an $\A^1$-weak equivalence.
	\end{corollary}
	\begin{remark}~
				 \begin{enumerate}
			
			\item If every fiber of $\pi$ is reduced and some fiber is reducible, then the algebraic space in Theorem \ref{factor-A1-fibration} is a non-separated scheme. In Example \ref{d-sphere}(2), the $\mathbb{A}^1$-fibration $\text{pr}_x: X \to \A^1_k$ factors through an \'etale locally trivial $\mathbb{A}^1$-bundle over the affine line with $d$-origins, where $d=deg(P(y))$. So, $X$ is $\mathbb{A}^1$-weakly equivalent to the affine line with $d$-origins (Corollary \ref{factorisation A1 equivalence}).
			\item If $\pi$ has a non-reduced fiber, then the algebraic space in Theorem \ref{factor-A1-fibration} is not a scheme. In Example \ref{d-sphere} (3), the $\mathbb{A}^1$-fibration $\text{pr}_x: X \to \A^1_k$ factors through an \'etale locally trivial $\mathbb{A}^1$-bundle over an algebraic space $\mathcal{C}$, which is the quotient of the affine line with double origin with free $\mathbb{Z}/2\mathbb{Z}$-action \cite[Example 12]{doub}. So, $X$ is $\A^1$-weakly equivalent to $\sC$ (Corollary \ref{factorisation A1 equivalence}). Thus in particular, $Pic(X)$ is isomorphic to $\Z/2\Z$.
		\end{enumerate}
	\end{remark}
By proceeding the same way as of proof of Theorem \ref{sing-etale-bundle}, the diagram \ref{lift} has a lift, if $\rho: X \to \sC$ is an vector bundle. Indeed, a vector bundle over $\mathbb{A}^n_\mathcal{O}$ corresponds to a $GL_n$-torsor over $\mathbb{A}^n_\mathcal{O}$ and the $\mathbb{A}^1$-invariance of the isomorphism class of vector bundles over affine base \cite[Theorem 1]{asokI} implies that a vector bundle over $\mathbb{A}^n_{\mathcal{O}}$ corresponds to a vector bundle over $\mathcal{O}$.  By Hilbert's Theorem 90, we have the isomorphism
$$H^1_{Zar}(\mathcal{O}, GL_n) \cong H^1_{Nis}(\mathcal{O}, GL_n) \cong H^1_{\text{\'et}}(\mathcal{O}, GL_n).$$
Finally in Zariski topology, a vector bundle (or equivalently, a projective module) over a local ring is trivial (or free) implies that diagram \ref{lift} in the proof of Theorem \ref{sing-etale-bundle} has a lift. Therefore, we have the following result:
\begin{theorem} \label{general bundle}
		Let $X$ be a smooth variety over $k$ and $\mathcal{C}$ be a smooth algebraic space over $k$. Suppose, $\rho: X \to \mathcal{C}$ is an  in \'etale, Nisnevich or Zariski topology. 
		Then the induced morphism $\rho_*: Sing_*(X) \to Sing_*(\mathcal{C})$ is a Nisnevich local weak equivalence.
	\end{theorem} 
%

  \section{Simplicial Contractibility of Koras-Russell Threefolds}
  The family of varieties known as Koras-Russell threefolds are the first examples of exotic algebraic structure on the 6-dimentional real manifold $\C^3$. They are diffeomorphic to $\C^3$, but not isomorphic to $\C^3$ as varieties. It is also remarkable that the Makar-Limanov invariant, i.e. ML($X$) is non-trivial, therefore $X \ncong \A^3_\C$ but ML($X\times \A^1_\C)=\C$, for the Koras-Russell threefolds $X$. Though it is not known whether $X \times \A^1_\C \cong \A^4_\C$. If it were true, i.e. $X \times \A^1_\C \cong \A^4_\C$, then $X$ has to be $\A^1$-contractible and $Sing_*(X)$ has to be in fact $\A^1$-fibrant as it is a retract of $Sing_*(\A^4_\C)$ which is $\A^1$-fibrant. Koras-Russell threefolds of the first kind $X$ are $\A^1$-contractible has already been proved in \cite[Theorem 3.1]{df}. Here we prove that $Sing_*(X)$ is $\A^1$-local.
 Before proving Theorem \ref{application purity koras}, we need the following technical lemma (Lemma \ref{key lemma}), which says that in case of open or closed immersions, the functor $Sing_*$ commutes with taking quotient of schemes (as Nisnevich sheaves).
\begin{lemma} \label{key lemma}
    Let $i:V \to U$ be a morphism between smooth varieties over $k$, which is either an open immersion or a closed immersion of schemes. Then the canonical morphism
    $$i_*: Sing_*(U)/Sing_*(V) \to Sing_*(U/V)$$
    is an isomorphism of simplicial sheaves.
\end{lemma}
\begin{proof}
There is a commutative diagram of simplicial sheaves
  \[
\xymatrix{
Sing_*(V) \ar[r] \ar[d] & Sing_*(U) \ar[d] \\
\text{*} \ar[r] & Sing_*(U/V)
}
\]
This induces a morphism of simplicial sheaves
$$i_*: Sing_*(U)/Sing_*(V) \to Sing_*(U/V).$$
We will show that for every essentially smooth Henselian local scheme $\mathcal{O}$ over $k$, the morphism
$$i_*: (Sing_*(U)/Sing_*(V))(\mathcal{O}) \to Sing_*(U/V)(\mathcal{O})$$
is an isomorphism of simplicial sets, that is for every $n \geq 0$ the morphism between $n$-simplices
$$i_*: (Sing_*(U)/Sing_*(V))(\mathcal{O})_n \to Sing_*(U/V)(\mathcal{O})_n$$
is a bijection.

Since over the Henselian local scheme, the presheaf section and sheaf section coincide, so
\begin{align*}
(Sing_*(U)/Sing_*(V))(\mathcal{O})_n & \cong Sing_*(U)(\sO)_n/Sing_*(V)(\sO)_n
\end{align*}
For any scheme $T$, the section $Sing_*(U)(T)_n/Sing_*(V)(T)_n$ is the pushout in the category of sets
 \[
\xymatrix{
Hom_{Sch/k}(\A^n_T, V) \ar[r] \ar[d] & Hom_{Sch/k}(\A^n_T, U) \ar[d] \\
\text{*} \ar[r] & Sing_*(U)(T)_n/Sing_*(V)(T)_n
}
\]
where $Sch/k$ is the category of $k$-schemes. So, $Sing_*(U)(T)_n/Sing_*(V)(T)_n$ is a pointed set canonically, denote the base point by $\bullet$. Also for the simplicial presheaf 
$$T \mapsto Sing_*(U)(T)/Sing_*(V)(T),$$
the restriction map
$$Sing_*(U)(T)/Sing_*(V)(T) to Sing_*(U)(T^\prime)/Sing_*(V)(T^\prime)$$
induced by a morphism $T^\prime \to T$ of $k$-schemes, respect the base points.

We will prove that the morphism  $i_*$ is a bijection between the pointed sets
$$i_*: ((Sing_*(U)/Sing_*(V))(\mathcal{O})_n, \bullet) \to (Sing_*(U/V)(\mathcal{O})_n, i_*(\bullet)),$$
for every Henselian local scheme $\sO$.



\textbf{Injectivity:} Suppose, $\phi, \psi \in (Sing_*(U)/Sing_*(V))(\mathcal{O})_n$ and 
$$i_*(\phi) = i_*(\psi) \in Sing_*(U/V)(\mathcal{O}).$$ 
For presheaf and the associated sheaf, the sections over $\mathcal{O}$ are same; thus $\phi, \psi \in (Sing_*(U)/Sing_*(V))^{\text{pre}}(\mathcal{O})$. So the sections are given by the morphisms $\phi, \psi: \mathbb{A}^n_{\mathcal{O}} \to U$ (again, denote by $\phi$ and $\psi$). 
Since $i_*(\phi) = i_*(\psi) \in (U/V)(\mathbb{A}^n_{\mathcal{O}})$, so there is a Nisnevich covering $f: W \to \mathbb{A}^n_{\mathcal{O}}$ such that 
$$\phi \circ f = \psi \circ f \in (U/V)^{\text{pre}}(W) = U(W)/V(W).$$
This means either $\phi \circ f, \psi \circ f: W \to V \hookrightarrow U$ or $\phi \circ f = \psi \circ f: W \to U$.
Since $Hom(-, U)$ is a Nisnevich sheaf, so for the second case we conclude $\phi = \psi: \mathbb{A}^n_{\mathcal{O}} \to U$. For the first case, since $f$ is a covering, in particular surjective; so the images of $\phi$ and $\psi$ are in $V$, thus both the morphisms $\phi$ and $\psi$ factor as
$$\phi: \mathbb{A}^n_{\mathcal{O}} \to V \hookrightarrow U \text{ and } \psi: \mathbb{A}^n_{\mathcal{O}} \to V \hookrightarrow U.$$
Hence 
$$\phi = \psi \in (Sing_*(U)/Sing_*(V))^{\text{pre}}(\mathcal{O}) = U(\mathbb{A}^n_{\mathcal{O}})/V(\mathbb{A}^n_{\mathcal{O}}).$$
So $i_*$ is injective.

\textbf{Surjectivity: } Suppose, $\alpha \in Sing_*(U/V)(\mathcal{O})_n$. So $\alpha \in (U/V)(\mathbb{A}^n_{\mathcal{O}})$. Thus there is a Nisnevich covering $W \to \mathbb{A}^n_{\mathcal{O}}$ such that $\alpha$ is given by an element $\phi \in (U/V)^{\text{pre}}(W)$ (that is, $\phi: W \to U$ is a morphism) such that there is a Nisnevich covering $\Tilde{W} \to W \times_{\mathbb{A}^n_{\mathcal{O}}} W$ such that 
$$\text{pr}_1^*(\phi)|_{\Tilde{W}} = \text{pr}_2^*(\phi)|_{\Tilde{W}} \in (U/V)^{\text{pre}}(\Tilde{W}) = U(\Tilde{W})/V(\Tilde{W}).$$
This means either 
$$\text{pr}_1^*(\phi)|_{\Tilde{W}}, \text{pr}_2^*(\phi)|_{\Tilde{W}} : \Tilde{W} \to V \hookrightarrow U$$
or
$$\text{pr}_1^*(\phi)|_{\Tilde{W}} = \text{pr}_2^*(\phi)|_{\Tilde{W}} : \Tilde{W} \to U.$$
For the second case, since $Hom(-, U)$ is a Nisnevich sheaf and the morphism $\Tilde{W} \to W \times_{\mathbb{A}^n_{\mathcal{O}}} W$ is a Nisnevich covering; so 
$$\text{pr}_1^*(\phi) = \text{pr}_2^*(\phi) : W \times_{\mathbb{A}^n_{\mathcal{O}}} W \to U$$
and again as $W \to \mathbb{A}^n_{\mathcal{O}}$ is a Nisnevich covering; so $\phi$ can be lifted to a morphism $\bar{\phi}: \mathbb{A}^n_{\mathcal{O}} \to U$, which proves surjectivity in this case. \par
For the first case, if $\text{pr}_1^*(\phi)|_{\Tilde{W}}, \text{pr}_2^*(\phi)|_{\Tilde{W}} : \Tilde{W} \to V$, then we claim that the morphism $\phi$ factors through $V$. Indeed, the composition of morphisms
$$\Tilde{W} \to W \times_{\mathbb{A}^n_{\mathcal{O}}} W \xrightarrow{\text{pr}_1} W$$
and 
$$\Tilde{W} \to W \times_{\mathbb{A}^n_{\mathcal{O}}} W \xrightarrow{\text{pr}_2} W$$
are Nisnevich coverings of $W$ and the images of $\text{pr}_1^*(\phi)|_{\Tilde{W}}$ and $\text{pr}_2^*(\phi)|_{\Tilde{W}}$ are in $V$, so the image of $\phi$ is also in $V$. Now we claim that $\alpha = i_*(\bullet) \in (U/V)(\mathbb{A}^n_{\mathcal{O}})$, where $\bullet$ is the distinguished element in $(Sing_*(U)/Sing_*(V))(\mathcal{O})$ mentioned before. This is equivalent to show that 
$$\alpha|_W = i_*(\bullet)|_W \in (U/V)(W),$$
where $W \to \mathbb{A}^n_{\mathcal{O}}$ is the Nisnevich covering and $-|_W$ denotes the restriction maps, induced by $W \to \A^n_\sO$, for both the presheaf $(U/V)^{\text{pre}}$, given by 
$$(U/V)^{\text{pre}}(T) = U(T)/V(T)$$
and the associated Nisnevich sheaf $U/V$.
Now, $\alpha|_W = \eta(\phi)$. Consider the commutative diagram
 \[
\xymatrix{
U(\A^n_\sO)/V(\A^n_\sO) \ar[r]^\eta \ar[d]^{|_W} & (U/V)(\A^n_\sO) \ar[d]^{|_W} \\
U(W)/V(W) \ar[r]_\eta & (U/V)(W)
}
\]
where 
$$\eta: (U/V)^{\text{pre}} \to U/V$$ 
is the sheafification morphism and in the square, the left vertical morphism respects the base point. So,
$$i_*(\bullet)|_W = \eta(\bullet|_W).$$
Since the image of $\phi$ is contained in $V$, so $\phi = \bullet|_W \in U(W)/V(W)$, so 
$$\alpha|_W =\eta(\bullet|_W)= i_*(\bullet)|_W \in (U/V)(W)$$ 
and consequently, $\alpha = i_*(\bullet)$. This proves the surjectivity. 

This completes the proof.
\end{proof}

\begin{theorem} \label{application purity koras}
    Suppose, $X$ is the Koras-Russell threefold of the first kind given by
    $$X \equiv \{x^mz = y^r +t^s+x\} \subset \mathbb{A}^4_k=\Spec\ k[x,y,z,t],$$
    where $m \geq 2$ and $r,s \geq 2$ are coprime integers. Then the structure map $Sing_*(X) \to \Spec\ k$ is a Nisnevich local weak equivalence or $Sing_*(X)$ is simplicially contractible and $Sing_*(X)$ is $\A^1$-local.
\end{theorem}
\begin{proof}
    We will follow the proof in \cite[Section 3.1]{df}. 
Let us write, 
$$X(s) \equiv \{x^mz = y^r +t^s+x\}; \ s \geq 1.$$
The morphism
$$\phi: X(1) \equiv \{x^mz = y^r+t+x\} \to \mathbb{A}^3_k \text{ defined as }(x,y,z,t) \mapsto (x, y ,z)$$
is an isomorphism.
Consider the closed immersion 
$$i: \mathbb{A}^2_k \equiv \Spec\ k[y,t] \to X(s) \text{ defined as } (y,t) \mapsto (-y^r-t^s, y,0, t).$$
Let $L$ be the line contained in $X(s)$:
$$L:=X(s) \cap \{x=y=0\}.$$
Then the closed immersion $i$ restricts to the morphism (we denote the restriction also by $i$)
$$i: \mathbb{A}^2_k \setminus \{(0,0)\} \to X(s) \setminus L.$$
Consider the following commutative diagram where each row is a $\A^1$-cofiber sequence
  \[
\xymatrix{
\mathbb{A}^2_k \setminus \{(0,0)\}  \ar[r] \ar[d]^i & \mathbb{A}^2_k \ar[d]^{i} \ar[r] & \mathbb{A}^2_k/(\mathbb{A}^2_k \setminus \{(0,0)\})  \ar[d]^q \\
X(s) \setminus L \ar[r] & X(s)  \ar[r] & X(s)/(X(s) \setminus L)  
}
\]
here $q$ is the morphism, induced by $i$. Now, applying the $Sing_*$ functor and by Lemma \ref{key lemma}, both rows of the following commutative diagram are also cofiber sequences:
\[
\xymatrix{
	Sing_*(\mathbb{A}^2_k \setminus \{(0,0)\})  \ar[r] \ar[d]^{i_*} & Sing_*(\mathbb{A}^2_k) \ar[d]^{i_*} \ar[r] & Sing_*(\mathbb{A}^2_k/(\mathbb{A}^2_k \setminus \{(0,0)\}))  \ar[d]^{q_*} \\
	Sing_*(X(s) \setminus L) \ar[r] & Sing_*(X(s))  \ar[r] & Sing_*(X(s)/(X(s) \setminus L))
}
\]
Thus to prove that $i_*: Sing_*(\A^2_k) \to Sing_*(X(s))$ is a simplicial weak equivalence, by very weak five lemma \cite[Lemma 2.1]{df} it suffices to prove that 
$$i_*: Sing_*(\A^2_k\backslash\{(0,0)\}) \to Sing_*(X(s)\backslash L) \text{ and } $$
$$q_*: Sing_*(\mathbb{A}^2_k/(\mathbb{A}^2_k \setminus \{(0,0)\}) \to  Sing_*(X(s)/(X(s) \setminus L))$$
are simplicial weak equivalences.
In \cite[Section 3]{df}, it was shown that the left vertical map 
$$i: \mathbb{A}^2_k \setminus \{(0,0)\} \to X(s) \setminus L$$ 
is an $\mathbb{A}^1$-weak equivalence. So to prove the map 
$$i_*: Sing_*(\A^2_k\backslash\{(0,0)\}) \to Sing_*(X(s)\backslash L)$$
is a simplicial weak equivalence, it suffices to show that
both the spaces $Sing_*(\A^2_k\backslash\{(0,0)\}) $ and  $Sing_*(X(s)\backslash L)$ are $\A^1$-local. 

The projection, 
$$SL_2 \equiv \{xw-yz=1 \subset \mathbb{A}^4_k\} \to \mathbb{A}^2_k \setminus \{0,0\} \text{ defined as } (x,y,z,w) \mapsto (x,y)$$
is an line bundle. Therefore, the induced map 
$$Sing_*(SL_2) \to Sing_*(\mathbb{A}^2_k \setminus \{0,0\})$$
is a simplicial weak equivalence (Proposition \ref{sing-etale-bundle}). Since, $Sing_*(SL_2)$ is $\mathbb{A}^1$-local \cite[Remark 3.3.8]{asokII}, so $Sing_*(\mathbb{A}^2_k \setminus \{0,0\})$ is also $\mathbb{A}^1$-local. 
By \cite[Proposition 3.1]{df}, there is a fourfold $W \in Sm/k$ such that $W$ admits Zariski locally trivial $\mathbb{A}^1$-bundles 
$$p_s: W \to X(s) \setminus L \text{ and } p_1: W \to X(1)\setminus L.$$
Therefore, the induced maps
$$(p_s)_*: Sing_*(W) \to Sing_*(X(s) \setminus L) \text{ and } (p_1)_*: Sing_*(W) \to Sing_*(X(1) \setminus L)$$
are Nisnevich local weak equivalences (Proposition \ref{sing-etale-bundle}).
 The projection map
$$p: X(1)\setminus L \to \mathbb{A}^2_k \setminus \{0,0\} \mathrm{\ defined\ as } (x,y,z,t) \mapsto (x,t) $$
is an line bundle. So again by Proposition \ref{sing-etale-bundle}, the induced map
$$p_*: Sing_*(X(1) \setminus L) \to Sing_*(\mathbb{A}^2_k \setminus \{0,0\})$$ 
is a Nisnevich local weak equivalence. Thus, $Sing_*(X(1) \setminus L)$ is $\A^1$-local and consequently \\
$Sing_*(X(s)\backslash L)$ is also $\mathbb{A}^1$-local. Thus the $\A^1$-weak equivalence $$i_*: Sing_*(\mathbb{A}^2_k \setminus \{0,0\}) \to Sing_*(X(s) \setminus L)$$ is a Nisnevich local weak equivalence.\\
Now we will prove that $q_*$ is a Nisnevich local weak equivalence.

Consider the projection
$$\rho: X(s) \equiv \{x^mz = y^r+t^s+x\} \to \mathbb{A}^3_k \equiv \Spec\ k[y,z,t]$$
$$ \text{ given by } (x,y,z,t) \mapsto (y,z,t).$$
Suppose, $f(x,y,z,t) = x^mz - y^r-t^s-x \in k[x,y,z,t]$ that defines $X(s)$. Firstly, we will find some open subset of $X(s)$, where $\rho$ is \'etale.
$$\nabla f = (mx^{m-1}z-1, -ry^{r-1}, x^m, -st^{s-1}).$$
Suppose, $\Tilde{x} = (x_0,y_0,z_0,t_0) \in X(s)$ is a $k$-point. So the tangent space $T_{\Tilde{x}}(X(s))$ is given by (as a vector subspace of $k^4$):
$$(mx_0^{m-1}z_0-1)x -ry_0^{r-1}y+ x_0^mz -st_0^{s-1}t = 0.$$
The induced map between the tangent spaces
$$\rho_*: T_{\Tilde{x}}(X(s)) \to k^3 \text{ is given by } (x, y,z,t) \mapsto (y,z,t).$$
The map $\rho_*$ is an isomorphism on the open subset $X(s) \backslash \{x^{m-1}z = \frac{1}{m}\}$ of $X(s)$. 

Thus $\rho|_{X(s) \backslash \{x^{m-1}z = \frac{1}{m}\}}: X(s) \backslash \{x^{m-1}z = \frac{1}{m}\} \to \mathbb{A}^3_k$ is \'etale.

Suppose, $V = X(s) \backslash (\{x^{m-1}z = \frac{1}{m}\} \cup \{x^{m-1}z = 1\})$. Then $V$ is an open subset of $X(s)$ and $L := X(s) \cap \{x = y=0\} \subset V$ and also $\rho|_V : V \to \mathbb{A}^3_k$ is an \'etale morphism. The inverse image $\rho|_V^{-1}(\{y=t=0\}) = L$. Let us denote the closed subscheme $\{y=t=0\}\subset \A^3$ also by $L$.
Consider the following elementary distinguished squares in the Nisnevich topology (in fact the left square is in the Zariski topology, all morphisms are open immersions):
 \[
\xymatrix{
V \setminus L \ar[r] \ar[d]^j & V \ar[d]^j & V \setminus L \ar[r] \ar[d]^{\rho|_V} & V \ar[d]^{\rho|_V} \\
X(s) \setminus L \ar[r] & X(s) & \A^3_k \backslash L \ar[r] & \A^3_k
}
\]
They induce the following isomorphisms of the Nisnevich sheaves
$$j: (V/(V \setminus L)) \cong (X(s)/(X(s) \setminus L)) \mathrm{\ and \ } (\rho|V)_*: (V/(V \setminus L)) \cong (\A^3_k/(\A^3_k\backslash L)).$$
Therefore, after applying the $Sing_*$ functor, we have the isomorphisms.
$$j_*: Sing_*(V/(V \setminus L)) \to Sing_*(X(s)/(X(s) \setminus L)) \mathrm{\ and \ }$$ $$(\rho|_V)_*: Sing_*(V/(V \setminus L)) \to Sing_*(\mathbb{A}^3_k/(\mathbb{A}^3_k \setminus L)).$$ 
Thus, the projection map $\rho: X(s) \to \mathbb{A}^3_k$ induces Nisnevich local weak equivalence 
$$\rho_*: Sing_*(X(s)/(X(s) \setminus L)) \to Sing_*(\mathbb{A}^3_k/(\mathbb{A}^3_k \setminus L)).$$
Therefore, to show that 
$$q_*: Sing_*(\mathbb{A}^2_k/(\mathbb{A}^2_k \setminus \{(0,0)\})) \to  Sing_*(X(s)/(X(s) \setminus L))$$ 
is a Nisnevich local weak equivalence, it is enough to show that the morphism
$$\rho \circ i: \mathbb{A}^2_k \to \mathbb{A}^3_k \text{ given by }(y,t) \mapsto (y,0,t)$$
induces Nisnevich local weak equivalence
$$Sing_*(\mathbb{A}^2_k/(\mathbb{A}^2_k \setminus \{(0,0)\})) \to Sing_*(\mathbb{A}^3_k/(\mathbb{A}^3_k \setminus L)).$$
Consider, the projection
$$\text{pr}: \mathbb{A}^3_k \to \mathbb{A}^2_k \text{ defined as }(y,z,t) \mapsto (y,t).$$
The morphisms $\text{pr}$ and $\text{pr}|_{\mathbb{A}^3_k \setminus L}: \mathbb{A}^3_k \setminus L \to \mathbb{A}^2_k \setminus \{(0,0)\}$ are line bundles. Therefore, by Proposition \ref{sing-etale-bundle}, the induced morphisms
$$\text{pr}_*: Sing_*(\A^3_k) \to Sing_*(\A^2_k) \text{ and }$$ 
$$(\text{pr}|_{\mathbb{A}^3_k \setminus L})_*: Sing_*(\mathbb{A}^3_k \setminus L) \to Sing_*(\mathbb{A}^2_k \setminus \{(0,0)\})$$
are Nisnevich local weak equivalences. The morphism $\rho \circ i$ is the zero section of $\text{pr}$, so the morphisms
$$Sing_*(\A^2_k) \to Sing_*(\A^3_k) \text{ and }Sing_*(\A^2_k \setminus \{(0,0)\}) \to Sing_*(\A^3_k \setminus L),$$
induced by $\rho \circ i$ are Nisnevich local weak equivalences. Therefore, by \cite[Lemma 2.11, Section 2.2]{mv}, the morphism 
$$Sing_*(\A^2_k)/(Sing_*(\A^2_k \setminus \{(0,0)\})) \to Sing_*(\A^3_k)/(Sing_*(\A^3_k \setminus L))$$
is a Nisnevich local weak equivalence. Finally using Lemma \ref{key lemma}, the morphism
$$Sing_*(\mathbb{A}^2_k/(\mathbb{A}^2_k \setminus \{(0,0)\})) \to Sing_*(\mathbb{A}^3_k/(\mathbb{A}^3_k \setminus L)),$$
induced by $\rho \circ i$, is a Nisnevich local weak equivalence. This completes the proof. Therefore, we conclude that $i_*: Sing_*(\A^2_k) \to Sing_*(X(s))$ is a simplicial weak equivalence, so $Sing_*(X(s))$ is simplicially contractible. Thus $Sing_*(X)$ is $\A^1$-local.

\end{proof}
Thus by using \cite[Lemma 4.1]{df}, we have the following corollary
\begin{corollary}
    Suppose, $X$ is the threefold contained in $\mathbb{A}^4_k = \Spec\ k[x,y,z,t]$ given by
    $$X \equiv x^mz = y^r+t^s+xq(x),$$
    where $m,r,s$ are integers, $m \geq 2$ and $r,s \geq 1$ are coprime and $q(x) \in k[x]$ is such that $q(0) \in k^*$. Then $Sing_*(X)$ is simplicially contractible.
\end{corollary}
	





\begin{remark}
Thus Theorem \ref{application purity koras} says that if $X$ is a Koras-Russell threefold of the first kind, then $Sing_*(X)(\mathcal{O})$ is a contractible  simplicial set, for any Henselian local scheme $\mathcal{O}$.
This in particular, implies that $Sing_*(X)$
is $\mathbb{A}^1$-local.
If the Koras-Russell threefold of the first kind were a counterexample to the Zariski cancellation in characteristics zero, i.e. $X \times_k \mathbb{A}^1_\C \cong \mathbb{A}^4_\C$, then for any $U \in Sm/k$, $Sing_*(X)(U)$ has to be a Kan fibrant and contractible simplicial set, as it is a retract of $Sing_*(\A^4_\C)(U)$ which is Kan fibrant and contractible (Example \ref{kan fibrant An} and Lemma \ref{contractible An}). So a natural question is whether $X$ is $\mathbb{A}^1$-naive in the sense of \cite[Definition 2.1.1]{asokII} or equivalently, $Sing_*(X)$ satisfies the affine Nisnevich excision \cite[Section 2.1]{asokI}.
\end{remark}


	\section{$\A^1$-invariance of $\pi_0^{\A^1}$ of smooth affine surfaces} \label{a1 invariance section}
In this section we will prove that for any smooth affine complex surface $X$, the $\mathbb{A}^1$-connected component sheaf $\pi_0^{\mathbb{A}^1}(X)$ is $\mathbb{A}^1$-invariant. For this let us recall the definition of $\mathbb{A}^1$-connected component sheaf first. 
\begin{definition} \cite[Section 3.2]{mv}
Suppose, $\sX$ is a space over $k$. The $\A^1$-connected component of $\sX$, denoted by $\pi_0^{\A^1}(\sX)$, is a Nisnevich sheaf of sets on $Sm/k$ associated to the presheaf 
$$U \in Sm/k \mapsto Hom_{\sH(k)}(U, \sX).$$
The space $\sX$ is called $\A^1$-connected, if $\pi_0^{\A^1}(\sX)$ is isomorphic to $\Spec\ k$. 
\end{definition}
\begin{remark}
\begin{enumerate}

\item  For a space $\sX$ the canonical morphism $\sX \to \pi_0^{\A^1}(\sX)$ is an epimorphism \cite[Corollary 2.1.5]{am}.
\item The space $\sX$ is $\A^1$-connected, if and only if $\pi_0^{\A^1}(\sX)(\Spec\ F)$ is trivial, for every finitely generated, separable field extension $F$ over $k$ \cite[Lemma 3.3.6]{ictp}.
\end{enumerate}
\end{remark}
\begin{definition}($\A^1$-invariant, see also Definition \ref{a1 local definition}) \label{invariant definition}  \hfill
	\begin{enumerate}
		\item	A presheaf of sets $\sF$ is called an $\A^1$-invariant if for any $U\in Sm/k$ the map
		$$\sF(U)\to \sF(U\times\A^1_k)$$ 
		induced by the projection $U\times\A^1_k \to U$ is a bijection of sets \cite[Definition 1.7]{mor}.
		\item A scheme $X$ over $k$ is called an $\A^1$-rigid, if $X$ is $\A^1$-invariant, as representable sheaf \cite[Example 2.4]{mv}. 
	\end{enumerate}
\end{definition}
\begin{remark} \cite[Lemma 2.16]{mazza} \label{mazza}
A presheaf of sets $\sF$  on $Sm/k$ is $\A^1$-invariant if and only if for every $U \in Sm/k$ the morphisms 
$$i_0^*, i_1^*: \sF(\A^1_k \times_k U) \to \sF(U),$$
induced by the $0$-section and $1$-section respectively, are equal.
\end{remark}
For a space $\sX$, the presheaf of sets $U \in Sm/k \mapsto Hom_{\sH(k)}(U, \sX)$ is $\A^1$-invariant, by the construction of $\sH(k)$. Morel conjectured that the sheaf $\pi_0^{\A^1}(\sX)$ is also $\A^1$-invariant \cite[Conjecture 1.12]{mor}. The Conjecture is false, in general. Ayoub constructed a space $\sX$ for which $\pi_0^{\A^1}(\sX)$ is not $\A^1$-invariant and this $\sX$ is not a representable sheaf \cite[Construction 8]{ayoub}.
However, $\pi_0^{\A^1}(\sX)$ is known to be $\A^1$-invariant for the following spaces $\sX$:
\begin{enumerate}
\item The space $\sX$ is $\A^1$-connected. 
\item The space $\sX$ is itself an $\A^1$-invariant Nisnevich sheaf or $\sX$ is an $\A^1$-rigid scheme (for example, abelian varieties, proper open subscheme of $\A^1_k$, a smooth projective curve of genus $g>0$) \cite[Example 2.1.10]{am}.
\item The space $\sX$ is a motivic $H$-group or $\sX$ is a homogeneous space for a motivic $H$-group \cite[Theorem 4.18]{u}.
\item The space $\sX$ is a smooth toric variety \cite[Lemma 4.2, Lemma 4.4]{w}.
\item The space $\sX$ is a smooth projective surface (over any field, if $\sX$ is a non-uniruled surface \cite[Corollary 3.15]{bhs} and over an algebraically closed field of characteristic zero, if $\sX$ is a ruled surface \cite[Theorem 1.2]{bs}).
\end{enumerate}

In this section, we will prove that if $X$ is an affine surface over an algebraically closed field of characteristic zero, then $\pi_0^{\A^1}(X)$ is an $\A^1$-invariant sheaf (Theorem \ref{connected-base-new}).

Let us first recall the construction and properties of the universal $\A^1$-invariant sheaf $\sL(\sF)$ from \cite{bhs}, associated to a Nisnevich sheaf of sets $\sF$ on $Sm/k$.
\begin{definition} \cite[Definition 2.9]{bhs}
Given a Nisnevich sheaf of sets $\sF$ on $Sm/k$, the sheaf $\sS(\sF)$ is defined to be the Nisnevich sheaf of sets associated to the presheaf $\sS^{\text{pre}}(\sF)$, which is defined as
$$\sS^{\text{pre}}(\sF)(U) = \sF(U)/\sim, \text{ for }U \in Sm/k,$$
where $\sF(U)/\sim$ the quotient corresponding to the equivalence relation ``$\sim$" on $\sF(U)$, generated by the naive $\A^1$-homotopies; briefly for $\alpha, \beta \in \sF(U)$, $\alpha \sim \beta$ if there are naive $\A^1$-homotopies $H_1, \dots, H_m \in \sF(U)$, satisfying $H_j\circ i_1 = H_{j+1}\circ i_0$, for every $j$, such that $H_1 \circ i_0 = \alpha$ and $H_m \circ i_1 = \beta$ (here $i_0, i_1: U \to \A^1_U$ are the $0$-section and the $1$-section respectively).

For an integer $n > 1$, the sheaf $\sS^n(\sF)$ is defined inductively
$$\sS^n(\sF) := \sS(\sS^{n-1}(\sF))$$
and $\sS^0(\sF) = \sF$. The quotient map $\sF(U) \to \sF(U)/\sim$ induces the canonical morphism
$$\eta: \sF \to \sS(\sF)$$
which is an epimorphism of Nisnevich sheaves. For every $n$, this $n$-th iteration $\sS^n$ defines a functor on the category of Nisnevich sheaves on $Sm/k$ and the canonical epimorphism $\sS^n(\sF) \to \sS^{n+1}(\sF)$ defines the natural transformation $\sS^n \to \sS^{n+1}$. The sheaf $\sL(\sF)$ is defined to be the filtered colimit with respect to the morphisms $\sS^n(\sF) \to \sS^{n+1}(\sF)$:
$$\mathcal{L}(\mathcal{F}) := \varinjlim_{n \geq 0}\mathcal{S}^n(\mathcal{F}).$$
There is an induced morphism $\sF \to \sL(\sF)$, which is an epimorphism.
\end{definition}
\begin{remark}
\begin{enumerate}
\item For $X \in Sm/k$, the sheaf $\sS(X)$ is the $\A^1$-chain connected component sheaf $\pi_0^{ch}(X)$, defined by Asok-Morel \cite[Definition 2.2.4]{am}; it is the Nisnevich sheaf associated to the presheaf $U \mapsto \pi_0(Sing_*(X)(U))$ \cite[Remark 2.10]{bhs}.
\item For a sheaf $\sF$, the canonical epimorphism $\sF \to \pi_0^{\A^1}(\sF)$ factors through the epimorphism $\sS(\sF) \to \pi_0^{\A^1}(\sF)$.
\item The sheaf $\sL(\sF)$ is $\A^1$-invariant \cite[Theorem 2.13]{bhs}.
\item For a sheaf $\sF$, both the sheaves $\pi_0^{\A^1}(\sF)$ and $\sL(\sF)$ satisfy the same universal property \cite[Remark 2.15]{bhs}: given a morphism $\sF \to \sG$ to an $\A^1$-invariant sheaf $\sG$, it factors uniquely through the canonical morphisms $\sF \to \pi_0^{\A^1}(\sF)$ and $\sF \to \sL(\sF)$. Since $\sL(\sF)$ is $\A^1$-invariant, this induces the morphism $\pi_0^{\A^1}(\sF) \to \sL(\sF)$, which is an epimorphism and  this morphism is an isomorphism if and only if $\pi_0^{\A^1}(\sF)$ is $\A^1$-invariant \cite[Corollary 2.18]{bhs}.
\item The canonical morphism $\pi_0^{\A^1}(\sF) \to \sL(\sF)$ induces the bijection over the sections $\Spec\ K$, for every finitely generated separable field extension $K/k$ (\cite[Corollary 2.2]{brs}, \cite[Corollary 2.15]{cb}).
\item If $X \in Sm/k$ is a proper, non-uniruled surface, then $\sS(X) \cong \sS^2(X)$ \cite[Theorem 3.14]{bhs}. We will prove (Theorem \ref{non-uniruled}) that if $X \in Sm/k$ is an affine, non $\A^1$-uniruled surface (Definition \ref{ruled-defn}), then also 
$$\sS(X) \cong \sS^2(X).$$
\end{enumerate}
\end{remark}
Let us recall the following two classes of affine varieties with dominant family of affine lines. Suppose, $X \in Sm/k$ is an affine variety, where $k$ is an algebraically closed field field of characteristic $0$.
\begin{definition}\label{ruled-defn}
\begin{enumerate}
\item $X$ is said to be \textbf{$\mathbb{A}^1$-uniruled or log-uniruled} if there is a dominant generically finite morphism $H: \mathbb{A}^1_k \times_k Y \to X$ for some $k$-variety $Y$. 
\item $X$ is said to be \textbf{$\mathbb{A}^1$-ruled} or affine-ruled if there is a Zariski open dense subset $U$ of $X$ such that $U$ is isomorphic to $\mathbb{A}^1_k \times_k Z$ for some $k$-variety $Z$ \cite[Section 2.1]{miyanishicrm}.
\end{enumerate}
\end{definition}
In case of smooth affine surfaces, we have the following equivalences
(\cite[Chapter 2, Theorem 2.1.1 and Chapter 3, Lemma 1.3.1]{miyanishicrm} and \cite[Lemma 1.8]{russell}):
\begin{theorem} \label{all uniruled equivalences}
Let $X$ be a smooth affine surface over an algebraically closed field $k$ of characteristic zero.
Then 
the following are equivalent:
\begin{enumerate}
	\item $X$ is $\mathbb{A}^1$-uniruled.
	\item $X$ has negative logarithmic Kodaira dimension.
	\item $X$ is $\mathbb{A}^1$-ruled.
	\item $X$ admits an $\mathbb{A}^1$-fibration $\pi: X \to C$.
\end{enumerate}
\end{theorem}
We will prove the $\A^1$-invariance of $\pi_0^{\A^1}(X)$, for a smooth affine surface $X$ over an algebraically closed field $k$ of characteristic zero, dividing in two classes: one is non $\A^1$-uniruled affine surfaces and the other one is $\A^1$-uniruled affine surfaces. Equivalently in terms of the logarithmic Kodaira dimension (denoted by $\bar{k}(X)$): one class consists such $X$ having $\bar{k}(X)$ is non-negative 
and the other class consists such $X$ having $\bar{k}(X)$ is negative.
	\subsection{$\A^1$-invariance of $\pi_0^{\A^1}(X)$, for non $\A^1$-uniruled, affine surface $X$:}
	
	Suppose, $X$ is a smooth affine non $\A^1$-uniruled surface over an algebraically closed field $k$ of characteristic $0$. In this subsection, we will prove that $\pi_0^{\A^1}(X)$ is $\A^1$-invariant, by proving $\sS(X) \cong \sS^2(X)$ (Theorem \ref{non-uniruled}, compare this with \cite[Theorem 3.14]{bhs} in case of proper, non-uniruled surfaces).

Suppose, $\alpha: \A^1_k \to X$ is a morphism, where $X \in Sm/k$ is affine. Then the image of $\alpha$ (set-theoretic image) is closed in $X$. Indeed, if the image $\text{Im}(\alpha)$ is non-constant, then we extend $\alpha$ to $\Tilde{\alpha}: \mathbb{P}^1_k \to \bar{X}$ ($\bar{X}$ is a smooth compactification of $X$). The morphism $\Tilde{\alpha}$ is a projective morphism, since $\Tilde{\alpha}$ factors as
$$\mathbb{P}^1_k \xrightarrow{\text{graph of } \Tilde{\alpha}} \mathbb{P}^1_k \times_k \bar{X} \xrightarrow{\text{projection}} \overline{X}.$$
So $\Tilde{\alpha}$ is a proper morphism and hence $\text{Im}(\Tilde{\alpha})$ is closed in $\bar{X}$. Since $X$ is affine, so any morphism from $\mathbb{P}^1_k$ to $X$ is constant; thus the point of infinity of $\mathbb{P}^1_k$ maps to a $k$-point in $\bar{X} \setminus X$. Therefore, $\text{Im}(\alpha) = \text{Im}(\Tilde{\alpha}) \cap X$ is closed in $X$. 

 Before proving the main theorem, we need the following technical lemma, which we will use in the proof.

\begin{lemma} \label{not contained}
Suppose, $\theta: U \to X$ is a morphism, where $U, X \in Sm/k$ and $X$ is affine. Assume that there is an $\alpha: \mathbb{A}^1_k \to X$ such that 
$\mathrm{Im}(\theta) \subset \mathrm{Im}(\alpha)$ (Notations \ref{notations}). Then $\theta = \alpha(0) \in \mathcal{S}(X)(U)$, where $\alpha(0)$ is considered as the morphism
$$U \to Spec  \ k \xrightarrow{\alpha(0)} X$$.
\end{lemma}
\begin{proof}
Since $\text{Im}(\theta)$ is contained in the image of $\mathbb{A}^1_k$, so there is a morphism $\Tilde{\theta}: U \to \mathbb{A}^1_k$ such that $\theta = \alpha \circ \Tilde{\theta}$. Indeed, if there is some irreducible component, say $U_0$ of $U$ such that its image is constant (a $k$-point), then $\theta|_{U_0}$ always factors through $\mathbb{A}^1_k$ (by taking a $k$-point from the fiber of $\alpha$ over that point), otherwise $\theta|_{U_0}: U_0 \to \text{Im}(\alpha)$ is dominant ($\text{Im}(\alpha)$ is a closed subscheme of $X$). 
In this case, $\theta|_{U_0}$ factors through the normalisation of $\text{Im}(\alpha)$, which is $\mathbb{A}^1_k$ itself. Therefore, $\theta$ always factors through $\mathbb{A}^1_k$. 
Consider the naive $\mathbb{A}^1$-homotopy $H: \mathbb{A}^1_U \to X$ defined as
$$(t,x) \mapsto \alpha(t\Tilde{\theta}(x)).$$
Then $H(0, x) = \alpha(0)$ and $H(1, x) = \theta(x)$.
Therefore, $\theta = \alpha(0) \in \mathcal{S}(X)(U)$.
\end{proof}
\begin{corollary} \label{not contained corr}
Suppose, $\theta, \gamma: U \to X$,
where $U, X \in Sm/k$ and $X$ is affine. Assume that there are $\alpha_1, \dots, \alpha_n : \mathbb{A}^1_k \to X$ such that $\alpha_i(1) = \alpha_{i+1}(0)$, for every $i$ and there are $p, q \leq n$ such that $\mathrm{Im}(\theta) \subset \mathrm{Im}(\alpha_p)$ and $\mathrm{Im}(\gamma) \subset \mathrm{Im}(\alpha_q)$. Then, $\theta = \gamma \in \mathcal{S}(X)(U)$.
\end{corollary}
\begin{proof}
Since $\text{Im}(\theta) \subset \text{Im}(\alpha_p) $ and $\text{Im}(\gamma) \subset \text{Im}(\alpha_q)$, so 
$$\theta = \alpha_p(0) \in \mathcal{S}(X)(U) \text{ and } \gamma = \alpha_q(0) \in \mathcal{S}(X)(U)$$
by Lemma \ref{not contained}. Therefore, $\theta =  \gamma \in \mathcal{S}(X)(U)$, since $\alpha_p$ and $\alpha_q$ are joined by a chain of $\mathbb{A}^1_k$'s.
\end{proof}

\begin{remark} (\cite[Remark 2.6]{choudhury2}) \label{ghost homotopy useful remark}
The morphism $\eta: X \to \sS(X)$ is an epimorphism. A homotopy $H \in \sS(X)(\A^1_U)$ (called an $\A^1$-ghost homotopy \cite[Definition 3.2]{bhs}) is given by a Nisnevich covering $f: V \to \A^1_U$ and a morphism $\phi: V \to X$ which makes the following diagram commutative:
\[
\xymatrix{V \ar[d]^f \ar[r]^{\phi} &X \ar[d]^{\eta}\\
\A^1_U \ar[r]^H &\sS(X)}
\]
Since, the sections
$$ \eta \circ \phi \circ \text{pr}_1 = \eta \circ \phi \circ \text{pr}_2 \in \sS(X)(V \times_{\A^1_U} V),$$
where $\text{pr}_1, \text{pr}_2: V \times_{\A^1_U} V \to V$ are the projections; so there is a Nisnevich covering $V^\prime \to  V \times_{\A^1_U} V$ along with a chain of homotopies
$$G_1, \dots, G_m : \A^1_{V^\prime} \to X \text{ such that } G_1(0)=\phi \circ \text{pr}_1|_{V^\prime} \text{ and }  G_m(1)=\phi \circ \text{pr}_2|_{V^\prime}$$
(Notations \ref{notations}). If the sections 
$$\phi \circ \text{pr}_1|_W = \phi \circ \text{pr}_2|_W: W \to X$$
are equal for some Nisnevich covering $W \to V \times_{\A^1_U} V$, then (since $X$ is a Nisnevich sheaf) there is a morphism $\Tilde{\phi}: \A^1_U \to X$ that lifts $\phi$, that is $\Tilde{\phi} \circ f = \phi$. Thus, $\eta \circ \Tilde{\phi} = H$. But this implies that 
$$H(0) = H(1) \in \sS(X)(U)$$ 
(Notations \ref{notations}). Therefore, if we assume that $H \in \sS(X)(\A^1_U)$ is a non-constant homotopy (Definition \ref{non-constant-homotopy-defn}), then for every Nisnevich covering 
$$W \to  V \times_{\A^1_U} V,$$
 the morphisms
$$\phi \circ \text{pr}_1|_{W} \neq \phi \circ \text{pr}_2|_{W}: W \to X.$$
So if $H$ is a non-constant homotopy, we can assume that $G_i: \A^1_{V^\prime} \to X$ is also a non-constant homotopy, for every $i$.
\end{remark}
We will quickly recall the following notations and terminologies from \cite{choudhury2}.
\begin{notation}\label{notations} \cite[Notations 2.8]{choudhury2} Let $\sF$ be a Nisnevich sheaf of sets in $Sm/k$, and $U\in Sm/k$.
\begin{enumerate}
\item Let $H \in \sF(\A^1_U)$ be a homotopy and $i_0,i_1: U \to \A^1_U$ be the 0-section and 1-section respectively. We denote the elements $H\circ i_0 \in \sF(U)$ by $H(0)$ and $H\circ i_1 \in \sF(U)$ by $H(1)$
respectively.
\item A chain of homotopies $H_1,\dots,H_n \in \sF(\A^1_U)$ means a collection of homotopies $H_i\in\sF(\A^1_U)$ such that $H_i(1)=H_{i+1}(0),$ for every $i$.
\item For a morphism $f: X \to Y$ of schemes, by $\text{Im}(f)$ we mean the set theoretic image of $f$, that is the set $f(X)$, which is a subset of $Y$. If the set $\text{Im}(f)$ consists more than one element, then we say $f: X \to Y$ is a non-constant morphism. Note that this is different from non-constant homotopy $H \in \sF(\A^1_U)$, which means that $H(0) \neq H(1) \in \sF(U)$
\item For a scheme $X$ over $k$, by ``there is a line in $X$" (or $X$ contains a line) we mean there exists a non-constant morphism $\gamma: \A^1_k \to X$. For a subset $Y \subset X$, by ``$Y$ contains a line" means that there is a line in $X$, given by $\gamma: \A^1_k \to X$ such that $Im(\gamma) \subset Y$.

\end{enumerate}
\end{notation}
\begin{definition}\label{non-constant-homotopy-defn} (\cite[Definition 2.4]{choudhury2}, see also \cite[Example 2.5]{choudhury2})
	Suppose, 
	$\sF$ is a Nisnevich sheaf of sets on $Sm/k$ and 
	$X, U \in Sm/k$. A homotopy $H \in \sF(\A^1_U)$ is said to be a non-constant homotopy, if the sections
	$$H (0) \neq H (1) \in \sF(U),$$
	
\end{definition}
\begin{remark} \label{line remark}
If there is a non-constant homotopy $H: \A^1_U \to X$ (Definition \ref{non-constant-homotopy-defn}), for some $X, U \in Sm/k$, $U$ is irreducible, then there is a line in $X$ (Notations \ref{notations}). Indeed, since $H: \A^1_U \to X$ is a non-constant, homotopy so $H(0) \neq H(1): U \to X$. Thus there is $x \in U(k)$ such that 
$$H(0)(x) \neq H(1)(x).$$
So the composition of morphisms
$$\gamma: \A^1_k \xrightarrow{id, x} \A^1_U \xrightarrow{H} X$$
is a line in $X$ (Notations \ref{notations}).
\end{remark}
\begin{theorem} \label{non-uniruled}
Let $k$ be an algebraically closed field of characteristic zero. Suppose that $X \in Sm/k$ is an affine surface, which is not $\mathbb{A}^1$-uniruled. Then the canonical morphism $\mathcal{S}(X) \to \mathcal{S}^2(X)$ is an isomorphism.
\end{theorem}
\begin{proof}
Suppose $U \in Sm/k$. The canonical morphism $\mathcal{S}(X) \to \mathcal{S}^2(X)$ is an epimorphism of Nisnevich sheaves (Remark \ref{ghost homotopy useful remark}). We will prove that the map
$$\mathcal{S}(X)(U) \to \mathcal{S}^2(X)(U)$$
is injective. Suppose $\alpha, \beta \in \mathcal{S}(X)(U)$ such that $\alpha = \beta \in \mathcal{S}^2(X)(U)$. If possible, assume that $\alpha \neq \beta \in \mathcal{S}(X)(U)$. Since $\alpha = \beta \in \mathcal{S}^2(X)(U)$, there is a Nisnevich covering $U^\prime \to U$ such that $\alpha|_{U^\prime} = \beta|_{U^\prime} \in \mathcal{S}^{\text{pre}}(\mathcal{S}(X))(U^\prime)$. Since $\alpha \neq \beta \in \mathcal{S}(X)(U)$, so $\alpha|_{U^\prime} \neq \beta|_{U^\prime} \in \mathcal{S}(X)(U^\prime)$. 
Thus there is an irreducible component $W$ of $U^\prime$ such that $\alpha|_W \neq \beta|_W \in \mathcal{S}(X)(W)$. Since, 
$$\alpha|_{U^\prime} = \beta|_{U^\prime} \in \mathcal{S}^{\text{pre}}(\mathcal{S}(X))(U^\prime),$$
so the sections $\alpha|_W = \beta|_W \in \mathcal{S}^{\text{pre}}(\mathcal{S}(X))(W)$. Therefore there is a non-constant homotopy $H \in \mathcal{S}(X)(\mathbb{A}^1_W)$ (Definition \ref{non-constant-homotopy-defn}) such that $H(0) = \alpha|_W$. 
In the next Proposition (Proposition \ref{ghost homotopy implies dominant 1}) we will prove that there is a dominant morphism $G: \mathbb{A}^1_Y \to X$, with $Y$ irreducible such that 
$$G(0, -) \neq G(1, -): Y \to X.$$ 
We can take here $dim(Y) = 1$. Thus we arrive at a contradiction, since $X$ is not an $\mathbb{A}^1$-uniruled surface. Hence the canonical map
$$\mathcal{S}(X) \to \mathcal{S}^2(X)$$
is an isomorphism.
Therefore, it remains to prove that if there is a non-constant homotopy $H \in \sS(X)(\A^1_U)$ (Definition \ref{non-constant-homotopy-defn}), then $X$ admits a dominant morphism $G: \A^1_W \to X$, which is also a non-constant homotopy (Definition \ref{non-constant-homotopy-defn}), for some $W \in Sm/k$ irreducible.

\end{proof}
The next proposition is similar to \cite[Proposition 2.9]{choudhury2}. 

\begin{proposition} \label{ghost homotopy implies dominant 1}
Let $X \in Sm/k$ be an 
affine surface and $U \in Sm/k$ be irreducible. 
Suppose that $H \in \mathcal{S}(X)(\mathbb{A}^1_U)$ is a non-constant homotopy (Definition \ref{non-constant-homotopy-defn}).
Then there is $W \in Sm/k$ irreducible, along with a dominant morphism $G: \mathbb{A}^1_W \to X$ such that $G(0,-) \neq G(1, -): W \to X$ (i.e $G$ is also a non-constant homotopy in the sense of Definition \ref{non-constant-homotopy-defn}). 
\end{proposition}
\begin{proof}
The homotopy $H \in \sS(X)(\A^1_U)$ is non-constant. Following Remark \ref{ghost homotopy useful remark}, 
there is a Nisnevich covering $V^\prime \to V \times_{\mathbb{A}^1_U} V$ and a chain of non-constant naive $\mathbb{A}^1$-homotopies (Notations \ref{notations}) $G_1, \dots, G_m : \mathbb{A}^1_{V^{\prime}} \to X$ such that 
$$G_1(0) = \phi \circ \text{pr}_1|_{V^\prime} \text{ and } G_m(1) = \phi \circ \text{pr}_2|_{V^\prime}.$$
We claim that $H(0)$ and $H(1)$ (Notations \ref{notations}) lift to $V$, over some Nisnevich covering $\Tilde{U} \to U$. This means that there is a Nisnevich covering $\Tilde{U} \to U$ and there are morphisms $\theta, \psi: \Tilde{U} \to V$ such that the following diagrams are commutative:
 
\[\xymatrixcolsep{3pc}
\xymatrix{\tilde{U}\ar[d] \ar[r]^{\theta} &V \ar[d]^f \ar[r]^{\phi} &X \ar[d]^{\eta} &\tilde{U}\ar[d] \ar[r]^{\psi} &V \ar[d]^f \ar[r]^{\phi} &X \ar[d]^{\eta}\\
U \ar[r]^{i_0} &\A^1_U \ar[r]^H &\sS(X) &U \ar[r]^{i_1} &\A^1_U \ar[r]^H &\sS(X)} 
\]

Indeed, by taking the pullback of $f$ with respect to $i_0, i_1: U \to \A^1_U$,
there are Nisnevich covering $U^\prime \to U$ and $U^{\prime \prime} \to U$ such that $H(0)|_{U^\prime}$ and $H(1)|_{U^{\prime \prime}}$ lift to $V$. Now taking a common refinement $\Tilde{U}$ of $U^\prime$ and $U^{\prime \prime}$, both the sections $H(0)|_{\Tilde{U}}$ and $H(1)_{\Tilde{U}}$ lift to $V$.

So,
$$H(0)|_{\Tilde{U}} =\eta \circ \phi \circ \theta \in \mathcal{S}(X)(\Tilde{U}) \text{ and } H(1)|_{\Tilde{U}} = \eta \circ \phi \circ \psi \in \mathcal{S}(X)(\Tilde{U}).$$
Since $H$ is a non-constant homotopy, so 
$$\eta \circ \phi \circ \theta \neq \eta \circ \phi \circ \psi \in \mathcal{S}(X)(\Tilde{U}).$$
Thus there is an irreducible component (which is also a connected component) $W$ of $\Tilde{U}$ such that 
$$(\eta \circ \phi \circ \theta)|_W \neq (\eta \circ \phi \circ \psi)|_W \in \mathcal{S}(X)(W).$$
Suppose, $(\phi \circ \theta)|_W = \theta_1 \in X(W)$ and $(\phi \circ \psi)|_W = \theta_2 \in X(W)$. The sections 
$$\eta \circ \theta_1 \neq \eta \circ \theta_2 \in \mathcal{S}(X)(W).$$
So by Corollary \ref{not contained corr}, there exists no chain of affine lines $\gamma_1, \dots, \gamma_n: \A^1_k \to X$ (Notations \ref{notations}) in $X$ such that $\text{Im}(\theta_1) \subset \text{Im}(\gamma_i)$ and $\text{Im}(\theta_2) \subset \text{Im}(\gamma_j)$, for some $i,j$.

Suppose, $V  = \coprod_{i =1}^n V_i$, where $V_i$'s are the irreducible components of $V$, which are also the connected components of $V$. Since $W$ is irreducible, $\text{Im}(\theta_1) \subset \overline{\phi(V_i)}$ and $\text{Im}(\theta_2) \subset \overline{\phi(V_j)}$, for some $i, j$ (Notations \ref{notations}(3)). 
We complete the proof considering the following cases.

\

\textbf{Case 1:} \textbf{Suppose, $\text{Im}(\theta_1), \text{Im}(\theta_2) \subseteq \overline{\phi(V_p)}$, for every $p$.}

  Since $H$ is a non-constant homotopy (Definition \ref{non-constant-homotopy-defn}), there is an irreducible component $V_0$ of $V^{\prime}$ such that (Remark \ref{ghost homotopy useful remark})
$$\phi \circ \text{pr}_1|_{V_0} \neq \phi \circ \text{pr}_2|_{V_0}: V_0 \to X.$$
We can assume that the naive $\mathbb{A}^1$-homotopy $G_1|_{\mathbb{A}^1_{V_0}}$ is non-constant.
Suppose, $V_0$ maps to $V_r \times_{\mathbb{A}^1_U} V_s$, for some $r, s \leq n$. Since the projection maps and the morphism $V^{\prime} \to V \times_{\mathbb{A}^1_U} V$ are the \'etale maps, so $\overline{\text{Im}(G_1|_{\mathbb{A}^1_{V_0}}(0))} = \overline{\phi(V_r)}$. 
Since $G_1|_{\mathbb{A}^1_{V_0}}$ is a non-constant homotopy (Notations \ref{notations}), so $\text{Im}(G_1|_{\mathbb{A}^1_{V_0}})$ contains a line, say $\gamma: \A^1_k \to X$ (Remark \ref{line remark}). Since $X$ is affine, so $\text{Im}(\gamma)$ is closed in $X$. If $\overline{\text{Im}(G_1|_{\mathbb{A}^1_{V_0}})} = \text{Im}(\gamma)$, then by our assumption in this case, $\text{Im}(\theta_1), \text{Im}(\theta_2) \subset \text{Im}(\gamma)$, since $\overline{\text{Im}(G_1|_{\mathbb{A}^1_{V_0}}(0))} = \overline{\phi(V_r)}$. This implies that $\theta_1 = \theta_2 \in \mathcal{S}(X)(W)$, by Corollary \ref{not contained corr} and it is a contradiction. Therefore, the non-constant homotopy $G_1|_{\mathbb{A}^1_{V_0}}$ is dominant.

\

\textbf{Case 2:} \textbf{Suppose there is some $p$ such that $\text{Im}(\theta_1) \nsubseteq \overline{\phi(V_p)}$.}

Case 2 consists two subcases:

     \textbf{Subcase 2(a):} \textbf{Renumbering if necessary, assume that there is some $\mathbf{r < n}$ such that 
$$\text{Im}(\theta_1), \text{Im}(\theta_2) \subseteq \overline{\phi(V_1)}, \dots, \overline{\phi(V_r)}$$ 
and $\text{Im}(\theta_1)  \nsubseteq \overline{\phi(V_p)}$, for some $p >r$.
}

There is an irreducible component $V_0$ of $V^\prime$ that maps to $V_r \times_{\mathbb{A}^1_U} V_p$. Since the projection maps and the morphism $V^{\prime} \to V \times_{\mathbb{A}^1_U} V$ are the \'etale maps, so 
$$\overline{\text{Im}(\phi \circ \text{pr}_1|_{V_0})} = \overline{\phi(V_r)} \text{ and } \overline{\text{Im}(\phi \circ \text{pr}_2|_{V_0})} = \overline{\phi(V_p)}.$$
Since $\text{Im}(\theta_1) \nsubseteq \overline{\phi(V_p)}$, so
$$\phi \circ \text{pr}_1|_{V_0} \neq \phi \circ \text{pr}_2|_{V_0}: V_0 \to X.$$
Thus we can assume that the naive $\mathbb{A}^1$-homotopy $G_1|_{\mathbb{A}^1_{V_0}}$ is a non-constant homotopy. So $\text{Im}(G_1|_{\mathbb{A}^1_{V_0}})$ contains a line $\gamma: \A^1_k \to X$ (Remark \ref{line remark}). Since $X$ is affine, so $\text{Im}(\gamma)$ is closed in $X$.  If $\overline{\text{Im}(G_1|_{\mathbb{A}^1_{V_0}})} = \text{Im}(\gamma)$, then by our assumption $\text{Im}(\theta_1), \text{Im}(\theta_2) \subset \text{Im}(\gamma)$, since $G_1|_{\mathbb{A}^1_{V_0}}(0) = \phi \circ \text{pr}_1|_{V_0}$. This implies that $\theta_1 = \theta_2 \in \mathcal{S}(X)(W)$, by Corollary \ref{not contained corr} and it is a contradiction. Therefore, the non-constant homotopy $G_1|_{\mathbb{A}^1_{V_0}}$ is dominant.

\
\textbf{Subcase 2(b):} \textbf{There exist $r,s$ such that 
$$\text{Im}(\theta_1) \subseteq \overline{\phi(V_r)}, \ \text{Im}(\theta_2) \subseteq \overline{\phi(V_s)} \text{ and } \text{Im}(\theta_2) \nsubseteq \overline{\phi(V_r)}.$$
}

There is an irreducible component $V_0$ of $V^\prime$ that maps to $V_r \times_{\mathbb{A}^1_U} V_s$. Since the projection maps and the morphism $V^{\prime} \to V \times_{\mathbb{A}^1_U} V$ are the \'etale maps, so 
$$\overline{\text{Im}(\phi \circ \text{pr}_1|_{V_0})} = \overline{\phi(V_r)} \text{ and } \overline{\text{Im}(\phi \circ \text{pr}_2|_{V_0})} = \overline{\phi(V_s)}.$$
Since $\text{Im}(\theta_2) \nsubseteq \overline{\phi(V_r)}$, so
$$\phi \circ \text{pr}_1|_{V_0} \neq \phi \circ \text{pr}_2|_{V_0}.$$
We can assume all the naive $\mathbb{A}^1$-homotopies $G_1|_{\mathbb{A}^1_{V_0}}, \dots, G_m|_{\mathbb{A}^1_{V_0}}$ are non-constant (Definition \ref{non-constant-homotopy-defn}). Since every $G_i|_{\mathbb{A}^1_{V_0}}$ is non-constant, so $\text{Im}(G_i|_{\mathbb{A}^1_{V_0}})$ contains a line, say $\gamma_i: \A^1_k \to X$ (Remark \ref{line remark}). Since $X$ is affine, $\text{Im}(\gamma)$ is closed in $X$. If for every $i$, $\overline{\text{Im}(G_i|_{\mathbb{A}^1_{V_0}})} = \text{Im}(\gamma_i)$, 
then there is a chain of affine lines $\gamma_1, \dots, \gamma_m: \A^1_k \to X$ such that (Notations \ref{notations})
$$\text{Im}(\theta_1) \subset \text{Im}(\gamma_1) \text{ and } \text{Im}(\theta_2) \subset \text{Im}(\gamma_m),$$
since
$$G_1|_{\mathbb{A}^1_{V_0}}(0) = \phi \circ \text{pr}_1|_{V_0} \text{ and } G_m|_{\mathbb{A}^1_{V_0}}(1) = \phi \circ \text{pr}_2|_{V_0}$$ 
and this would imply that $\theta_1 = \theta_2 \in \mathcal{S}(X)(W)$, by Corollary \ref{not contained corr}. It is a contradiction. Therefore, there is some $j$ such that $\overline{\text{Im}(G_j|_{\mathbb{A}^1_{V_0}})}$ properly contains a line. Thus the non-constant homotopy $G_j|_{\mathbb{A}^1_{V_0}}$ is dominant. 

\

\textbf{Case 3:} \textbf{Suppose there is some $q$ such that $\text{Im}(\theta_2) \nsubseteq \overline{\phi(V_q)}$.}

Case 3 is similar to Case 2. Hence, in this Case also, there exists a non-constant homotopy $G: \A^1_Y \to X$ (Definition \ref{non-constant-homotopy-defn}), for some $Y \in Sm/k$ irreducible such that $G$ is a dominant morphism.

Therefore, the Proposition is proved.
\end{proof}

\begin{remark}
If $X$ is not an $\mathbb{A}^1$-uniruled surface, then in each cases $\overline{\text{Im}(G_i)}$ has dimension $1$. Since the image of an affine line is closed in $X$, so $\overline{\text{Im}(G_i)}$ is the image of an affine line. Thus $\theta_1 = \theta_2 \in \mathcal{S}(X)(W)$.
\end{remark}

\begin{corollary} \label{termination}
Suppose, $X \in Sm/k$ is an affine surface, which is not an $\A^1$-uniruled surface. Then the canonical epimorphism 
$$\mathcal{S}(X) \to \mathcal{L}(X)$$
is an isomorphism.
\end{corollary}
\begin{proof}
Since the canonical morphism $\mathcal{S}(X) \to \mathcal{S}^2(X)$ is an isomorphism by Theorem \ref{non-uniruled}, so all the canonical morphisms $\mathcal{S}^n(X) \to \mathcal{S}^{n+1}(X)$ is an isomorphism, for every $n$. 
Since the sheaf $\mathcal{L}(X) =  \underset{n}{\varinjlim} \; \mathcal{S}^n(X)$, 
so the canonical morphism
$$\mathcal{S}(X) \to \mathcal{L}(X)$$
is an isomorphism. 
\end{proof}
\begin{corollary} \label{invariant non-uniruled}
Let $k$ be an algebraically closed field of characteristic $0$. Suppose that $X \in Sm/k$ is an affine and not an $\mathbb{A}^1$-uniruled surface. Then $\pi_0^{\mathbb{A}^1}(X)$ is $\mathbb{A}^1$-invariant.
\end{corollary}
\begin{proof}
The canonical epimorphism $\mathcal{S}(X) \to \pi_0^{\mathbb{A}^1}(X)$ gives an epimorphism $\Phi: \mathcal{L}(X) \to \pi_0^{\mathbb{A}^1}(X)$, by Corollary \ref{termination}. There is a canonical epimorphism $\Psi: \pi_0^{\mathbb{A}^1}(X) \to \mathcal{L}(X)$. Since $\mathcal{L}(X)$ is the universal $\mathbb{A}^1$-invariant sheaf, so $\Psi \circ \Phi$ is identity. Hence $\Phi$ is an isomorpism.
Therefore,
$$\pi_0^{\mathbb{A}^1}(X) \cong \mathcal{L}(X) \cong \mathcal{S}(X)$$
and consequently, $\pi_0^{\mathbb{A}^1}(X)$ is $\mathbb{A}^1$-invariant.
\end{proof}
\begin{corollary}
Let $X \in Sm/k$ be an affine surface, where $k$ is an algebraically closed field of characteristic zero. Suppose that $X$ has non-negative logarithmic Kodaira dimension. Then $\pi_0^{\mathbb{A}^1}(X)$ is $\mathbb{A}^1$-invariant. For example, if $X$ is the Ramanujan surface \cite[Section 3]{ra} or $X$ is a tom Dieck-Petrie surface, then $\pi_0^{\mathbb{A}^1}(X)$ is $\mathbb{A}^1$-invariant.
\end{corollary}
\begin{proof}
By \cite[Lemma 1.8]{russell} an $\mathbb{A}^1$-uniruled affine surface has negative logarithmic Kodaira dimension. So the corollary follows from Corollary \ref{invariant non-uniruled}.
\end{proof}
\begin{remark}
Corollary \ref{invariant non-uniruled} gives an affirmative answer to \cite[Conjecture 2.2.8]{am}, in case of affine, non $\A^1$-uniruled, smooth surfaces $X$. However, in general the Conjecture is known to be false \cite[Section 4]{bhs}. Also, note that if $Sing_*(X)$ is $\mathbb{A}^1$-local, then the epimorphism 
$$\mathcal{S}(X) \to \pi_0^{\mathbb{A}^1}(X)$$
is an isomorphism \cite[Remark 2.2.9]{am}. Thus it is natural to ask for an affine non $\A^1$-uniruled, smooth surface $X$, whether $Sing_*(X)$ is $\mathbb{A}^1$-local.
\end{remark}
	
\subsection{$\A^1$-invariance of $\pi_0^{\A^1}(X)$, for $\A^1$-uniruled, affine surface $X$:}
\

Let $X$ be a smooth affine surface over an algebraically closed field $k$ of characteristic zero. In this subsection, we will prove that if $X$ is an $\A^1$-uniruled surface (Definition \ref{ruled-defn}), then $\pi_0^{\A^1}(X)$ is $\A^1$-invariant (Theorem \ref{connected-base-new}). Recall that, $X$ is $\A^1$-uniruled if and only if $X$ admits an $\A^1$-fibration $\pi:X \to C$, onto a smooth curve $C$ (Theorem \ref{all uniruled equivalences}) and the $\A^1$-fibration $\pi$ factors as $\pi= \alpha \circ \theta$ (Theorem \ref{factor-A1-fibration}), where $\theta: X \to \sC$ is an \'etale locally trivial $\A^1$-bundle over an algebraic space $\sC$ and $\alpha: \sC \to C$ is the structure morphism which is surjective, quasi-finite and birational. We have proved that the morphism $\theta$ is an $\A^1$-weak equivalence (Corollary \ref{factorisation A1 equivalence}). The $\A^1$-fibration $\pi: X \to C$ can be divided into two following classes -

\begin{enumerate}
	\item $C$ is an $\A^1$-rigid curve
	\item $C$ is not $\A^1$-rigid, hence $\A^1$-connected curve.
\end{enumerate} 
 We will prove the $\A^1$-invariance of $\pi_0^{\A^1}(X)$ by proving $\pi_0^{\A^1}(\sC)$ is $\A^1$-invariant in each case.

First we consider the case of $\A^1$-fibration $\pi: X\to C$ onto an $\A^1$-rigid smooth curve $C$. The following theorem says that if the base curve $C$ is $\mathbb{A}^1$-rigid, then any morphism $\mathbb{A}^1_k \to \sC$ is constant, which means the morphism factors through $\Spec k$. 

		\begin{theorem}\label{rigid-base}
			Let $X$ be a smooth affine surface, which admits an $\mathbb{A}^1$-fibration $\pi: X \to C$ onto an  
$\A^1$-rigid, smooth curve $C$ and
$\sC$ be the algebraic space 
appearing in the factorisation of $\pi$, by Theorem \ref{factor-A1-fibration}. Then any morphism $H: \A^1_k \to \sC$, factors through $\Spec k$.
		\end{theorem}
		
		\begin{proof}
			Suppose, the $\A^1$-fibration $\pi: X \to C$ factorises as (Theorem \ref{factor-A1-fibration})
$$X \xrightarrow{\theta} \mathcal{C}\xrightarrow{\alpha} C,$$ 
where $\sC$ is an algebraic space along with the morphism $\alpha$, which is surjective, quasi-finite, birational morphism to $C$ and $\theta$ is an étale $\A^1$-bundle. 

Suppose, $H: \A^1_k \rightarrow \sC$ is a morphism. Since $C$ is $\mathbb{A}^1$-rigid, so the morphism $\alpha \circ H: \mathbb{A}^1_k \to C$  
factors as
$$\A^1_k \xrightarrow{\pr} \Spec\ k  \xrightarrow{x} C,$$
we have the following commutative diagram 

\[
\xymatrix{
\A^1_k \ar[r]^H \ar[d]^{\pr}& \sC \ar[d]^\alpha\\
\Spec\ k \ar[r]^x & C
}\]

              Suppose, $x: \Spec\ k \to C$ maps to a $k$-point of $C$ (as it cannot map to the generic point of $C$). Since $\alpha$ is quasi-finite, so the fiber $\alpha^{-1}(x)$ consists only finitely many $k$-points. If Im$(H)\subset \sC$ is a subscheme, then $H$ is constant, since the topological space $\overline{\text{Im}(H)}$ is irreducible and connected . Otherwise Im$(H)$ must maps to a non-reduced point $\mathfrak{c}$ of $\sC$ (a point of $\sC$ over which the fiber is non-reduced) with $\kappa(\mathfrak{c})=k[t]/(t^m)$ for some $m>1$ . Then we have an étale covering $\phi: U \to \sC$ which have an unique $k$-point $y\mapsto \mathfrak{c}$ and $H$ factors through $U$. Since the only $k$-algebra map between $\kappa(\mathfrak{c})=k[t/(t^m)]\to \kappa(y)=k$ is $t\mapsto 0$ and $H$ factors through $U$, it must be constant and factors through $\Spec k$. 

		\end{proof}

Before proving the $\A^1$-invariance of $\pi_0^{\A^1}(X)$,
we need the
 this following technical result, which says that the algebraic space $\sC$ in Theorem \ref{factor-A1-fibration} is $\A^1$-invariant if and only if any morphism $\mathbb{A}^1_k \to \sC$ is constant. 

\begin{lemma}\label{not-a1-rigid}
	Let $X$ be a smooth affine surface, which admits an $\mathbb{A}^1$-fibration $\pi: X \to C$ onto a  
	smooth curve $C$ and
	$\sC$ is the algebraic space 
	in Theorem \ref{factor-A1-fibration}. If $\sC$ is not $\A^1$-invariant, then there exists a non-constant morphism $\gamma:\A^1_k \to \sC$.
\end{lemma}

\begin{proof}
The $\A^1$-fibration $\pi: X \to C$
from a smooth affine surface $X$ to a smooth curve $C$ factors through an \'etale locally trivial $\A^1$-bundle $\theta: X \to \sC$
over a smooth algebraic space $\sC$ (Theorem \ref{factor-A1-fibration}). Assume that $\sC$ is not $\A^1$-invariant.
Then by \cite[Lemma 2.16]{mazza}, there is $H\in \sC(\A^1_U)$ for some $U \in Sm/k$ such that 
$H(0) \neq H(1) \in \sC(U)$ (Notations \ref{notations}). Since $\sC$ is an \'etale sheaf on $Sm/k$, we can moreover assume $U$ to be affine.
The pullback $\theta^*H: \A^1_U\times_\sC X \to \A^1_U$ is an \'etale $\A^1$-bundle over $\A^1_U$.
An étale $\A^1$-bundle over $\A^1_U$ corresponds to a cohomology class in $H^1_{\text{\'et}}(\A^1_U,\Aut(\A^1_k))$ and 
$$H^1_{\text{\'et}}(\A^1_U,\Aut(\A^1_k)) \cong H^1_{\text{Zar}}(\A^1_U,\Aut(\A^1_k)) \cong H^1_{\text{Zar}}(U,\Aut(\A^1_k)),$$ 
so we can choose a Zariski affine trivialisation of $\{U_i\}_i$ of $U$ such that $\theta^*H|_{\A^1_{U_i}}$ is trivial, for every $i$. Thus there is some $i,\ H|_{\A^1_{U_i}}(0) \neq H|_{\A^1_{U_i}}(1) \in \sC(U_i)$. Therefore, 
we can assume that there is $H \in \sC(\A^1_U)$, where $U \in Sm/k$ is irreducible, affine and $H$ satisfies $H(0) \neq H(1) \in \sC(U)$ and also the pullback $\theta^*H: \A^1_U\times_\sC X \to \A^1_U$ is a trivial $\A^1$-bundle.
Consider the following diagram:	
  
  \begin{equation}\label{main-lemma-diagram}
  	\xymatrix{
  		H^{-1}(V)\times_\sC X \ar@{^{(}->}[r] \ar[d]^{\theta^*H}&\A^1_U\times_\sC X \ar[r]^-{H'} \ar[d]_{\theta^*H}& X \ar[d]^\theta & V\times_\sC X \ar@{_{(}->}[l] \ar[d]_\theta\\
  		H^{-1}(V) \ar@{^{(}->}[r] \ar@/^/[u]^{s_0}&\A^1_U \ar@/_/[u]_{s_0} \ar[r]^H&\sC & V \ar@{_{(}->}[l] \ar@/_/[u]_{s_0}
  	}
  \end{equation}
  
Since the $\A^1$-fibration $\pi$ is generically trivial, there is an open subspace $V\subset\sC$
, which is also a scheme,
over which the bundle is trivial, that is $\theta: V \times_\sC X \to V$ is a trivial $\A^1$-bundle. We denote all the zero sections of corresponding bundles by $s_0$. So the following diagram is a pullback diagram

  \begin{equation}\label{main-lemma-diagram-pulback}
  \xymatrixcolsep{7pc}
  \xymatrix{
  H^{-1}(V)\times_\sC X \cong H^{-1}(V)\times \A^1_k \ar[d]^{\theta^*H} \ar[r]^-{(H,id)}& V\times \A^1_k \cong V\times_\sC X \ar[d]_{\theta} \\
  H^{-1}(V)\ar[r]^H \ar@/^/[u]^{s_0} & V \ar@/_/[u]_{s_0}}
  \end{equation}

Therefore by commutativity of diagram (\ref{main-lemma-diagram}) we have 
$$(H^\prime \circ s_0)(H^{-1}(V)) = s_0(H(H^{-1}(V))).$$
Since, $H$ is a non-constant homotopy (that is $H(0) \neq H(1) \in \sC(U)$) and the algebraic space $\sC$ is of dimension $1$, so the restriction
$$H: H^{-1}(V) \to V$$
is dominant. So,
$$\overline{(H'\circ s_0)(H^{-1}(V))}= \overline{s_0(V)} \subseteq V\times_\sC X$$

 Thus, 
dim$\overline{(H'\circ s_0)(H^{-1}(V))}= \dim\overline{(s_0(V))}=1$, as closures in $V\times_\sC X$.

 Let $\widetilde{H}$ be the composition of the morphisms
$$ \A^1_U\xrightarrow{s_0}\A^1_U\times_\sC X \xrightarrow{H'} X.$$
Then, $\widetilde{H}$ is a lift of $H$, that is $\theta \circ \widetilde{H} =H$.

 \begin{figure}
	
	\includegraphics[width=0.5\textwidth]{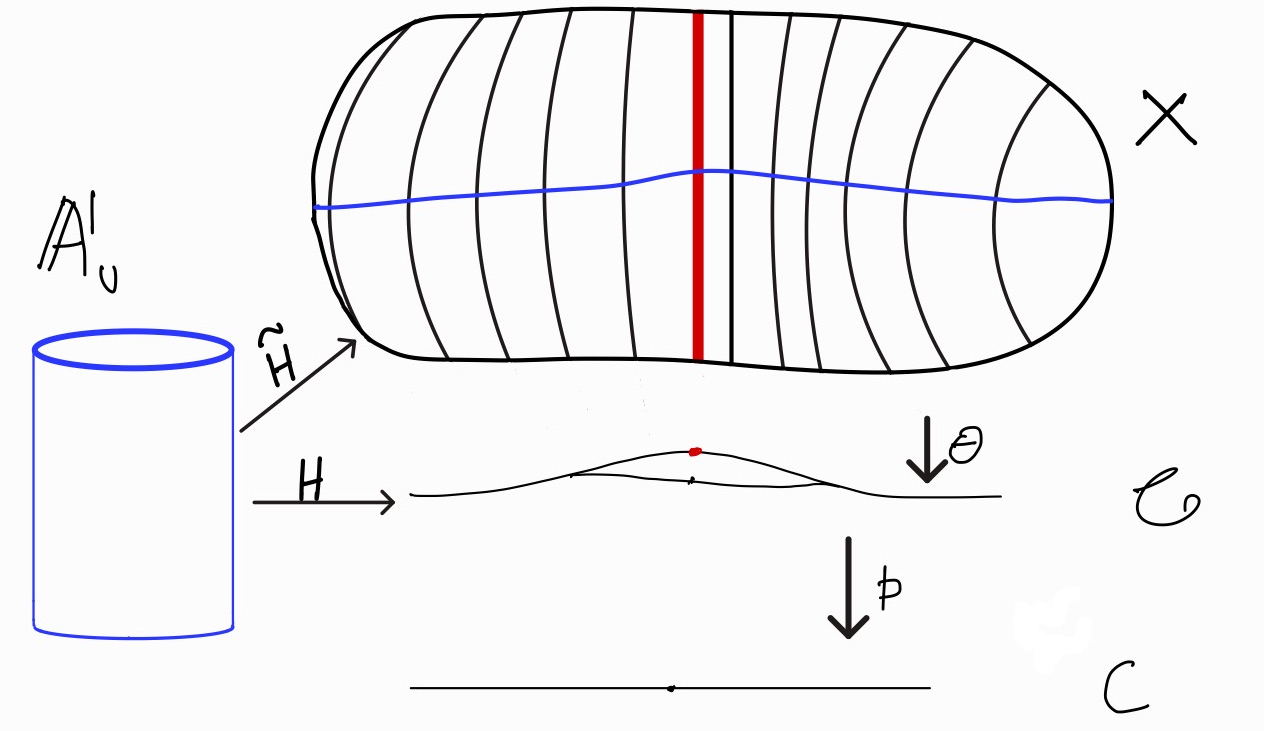}
	\caption{Non-constant map from $\A^1_k$}
\end{figure}
\graphicspath{{./Surface-2.jpg}}
Since, $\overline{(H'\circ s_0)(H^{-1}(V))}$ as closure in $V\times_\sC X$ is of dimension $1$ and $V\times_\sC X$, $H^{-1}(V)$ are open subschemes of $X$, $\A^1_U$
respectively, so
$$\dim\overline{\widetilde{H}(\A^1_U)}= 1, \text{ as closure in } X.$$ 
Since the sections $H(0) \neq H(1) \in \sC(U)$ and $\widetilde{H}$ is a lift of $H$,
so, 
$$\widetilde{H}(0) \neq \widetilde{H}(1) \in X(U).$$

Thus there exists $x \in U(k)$ such that $\widetilde{H}(0)(x)\neq \widetilde{H}(1)(x)$. Now define $\widetilde{\gamma}: \A^1_k \to X$ as
$$\widetilde{\gamma}: \A^1_k \times \Spec(k) \xrightarrow{(id\times x)} \A^1_k \times U \xrightarrow{\widetilde{H}} X.$$


 Since $X$ is affine, so the image $\widetilde{\gamma}(\A^1_k)$ is closed in $X$ and since $\widetilde{\gamma}(0) \neq \widetilde{\gamma}(1)$, so $\dim \widetilde{\gamma}(\A^1_k) = 1$. Now $\widetilde{\gamma}(\mathbb{A}^1_k) \subseteq \widetilde{H}(\mathbb{A}^1_U)$ and since, $\dim\overline{\widetilde{H}(\A^1_U)}= 1$, so
 $$\overline{\widetilde{H}(\A^1_U)} =   \widetilde{\gamma}(\mathbb{A}^1_k) = \widetilde{H}(\mathbb{A}^1_U).$$
 Since, $H(0) \neq H(1) \in \sC(U)$, so the image $\widetilde{H}(\mathbb{A}^1_U)$ 
 can not be contained in a single fiber over $\theta$.
Therefore, $\gamma:=\theta \circ \widetilde{\gamma}:\A^1_k \to \sC$ is a non-constant morphism.
 
\end{proof}

Therefore, by Theorem \ref{rigid-base} and Lemma \ref{not-a1-rigid}, the algebraic space $\sC$ is also $\A^1$-invariant in this case.
\begin{corollary} \label{a1 invariance of algebraic space}
Let $X$ be a smooth affine surface, which admits an $\mathbb{A}^1$-fibration $\pi: X \to C$ onto an  
$\A^1$-rigid, smooth curve $C$ and
$\sC$ be the algebraic space 
in Theorem \ref{factor-A1-fibration}. Then $\sC$ is also $\A^1$-invariant.
\end{corollary}
The $\A^1$-invariant sheaves are fibrant objects in the $\A^1$-model structure. So, if $\sC$ is $\A^1$-invariant, then $\pi_0^{\A^1}(\sC) \cong \sC$. Thus by Corollary \ref{factorisation A1 equivalence} we have the following:
\begin{corollary} \label{rigid case a1 invariant}
Let $X$ be a smooth affine surface, which admits an $\mathbb{A}^1$-fibration $\pi: X \to C$ onto an  
$\A^1$-rigid, smooth curve $C$. Then $\pi_0^{\A^1}(X) \cong \sC$, where $\sC$ is the algebraic space in Theorem \ref{factor-A1-fibration}. Thus, in particular $\pi_0^{\A^1}(X)$ is $\A^1$-invariant.
\end{corollary}
Since two $\A^1$-invariant sheaves are isomorphic if and only if they are $\A^1$-weakly equivalent, so by
Corollary \ref{factorisation A1 equivalence} and Corollary \ref{a1 invariance of algebraic space}, we can describe the $\A^1$-homotopy types of $X$, if $C$ is an $\A^1$-rigid curve.
\begin{corollary} \label{homotopy types base rigid}
Suppose, $X_i$'s are two $\A^1$-uniruled, smooth, affine surfaces that admit $\A^1$-fibrations $\pi_i: X_i \to C_i$
onto $\A^1$-rigid smooth curves $C_i$'s respectively, for $i=1,2$. Suppose that $\pi_i$ factors through $\theta_i: X \to \sC_i$, over smooth algebraic spaces $\sC_i$ respectively (as appeared in Theorem \ref{factor-A1-fibration}), for $i=1,2$.
Then $X_1$ and $X_2$ are $\A^1$-weakly equivalent if and only if $\sC_1$ and $\sC_2$ are isomorphic.
\end{corollary}
\begin{remark}
More generally, the conclusion of Corollary \ref{homotopy types base rigid} is true, if we only assume that the algebraic spaces $\sC_1, \sC_2$ are $\A^1$-invariant. Note that by Corollary \ref{not-a1-rigid}, if the base curve $C$ of the $\A^1$-fibration $\pi:X \to C$ is $\A^1$-rigid, then the algebraic space $\sC$ is also $\A^1$-invariant. But there exists $X$ with an $\A^1$-fibration $X \to \A^1_k$ such that the algebraic space $\sC$ (by Theorem \ref{factor-A1-fibration}) is $\A^1$-invariant (Example \ref{rigid example}).
\end{remark}
		\begin{example}
			\begin{enumerate}
				\item Let $X$ be the surface
                $$X\equiv\{(x-1)^2z=y(y-1)+(x-1)\}\backslash \{z=y(y-1)-1\}.$$
                $X$ is smooth by the Jacobian criterion and let $\pi=\pr_x$ be the $\A^1$-fibration to $\G_m=\A^1\backslash\{0\}$, an $\A^1$-rigid curve. In this case, the algebraic space $\sC$ is actually a scheme and it is $\G_m$ with double $1$, let us denote it by $\tilde{\G}_m$.\par $\tilde{\G}_m$ is $\A^1$-rigid as for any morphism $f:\A^1\to\tilde{\G}_m$ the composition morphism $$\A^1\xrightarrow{f}\tilde{\G}_m\xrightarrow{q}\G_m$$gives a constant map, thus $f$ must be constant. Here $q$ sends both of the $1$ to same point $1$ in $\tilde{\G}_m$ and identity otherwise. 
				\item Let $X\equiv\{(x-1)^2z=y^2+(x-1)\}\backslash \{z=y^2\} \xrightarrow{\pr_x}\G_m$ be the $\A^1$-fibration. It factors as $X\xrightarrow{\rho}\sC\xrightarrow{p}\G_m$. Then similarly any $f:\A^1\to \sC$ the composition $p \circ f$ is constant. Since $p$ is surjective $f$ must be constant.
			\end{enumerate}
		\end{example}

\begin{definition} \label{multi-section definition}
Suppose, $\pi: X \to C$ is an $\mathbb{A}^1$-fibration. A morphism $\gamma: C \to X$ is called a horizontal section (or a multi-section) of $\pi$, if the morphism $\pi \circ \gamma: C \to C$ is surjective. The multi-section $\gamma$ is said to be a section of $\pi$, if the morphism $\pi \circ \gamma: C \to C$ is the identity morphism on $C$.
\end{definition}	
\begin{remark} \label{multi-section remark}
\begin{enumerate}
\item A multi-section $\gamma: C \to X$ of an $\A^1$-fibration $\pi: X \to C$, intersects every fiber $\pi^{-1}(x)$ , for all $x \in C$. However, a multi-section does not always intersect every irreducible components of a degenerate fiber (Example \ref{multi-section example}(3)).
\item Suppose, the base of the $\mathbb{A}^1$-fibration $\pi$ is $\mathbb{A}^1_k$.  Then in Definition \ref{multi-section definition}, the surjectivity of $\pi \circ \gamma$ is equivalent to $\pi \circ \gamma$ is non-constant, since the complement of finitely many points (at least one) in $\mathbb{A}^1_k$ is $\mathbb{A}^1$-rigid. If there exists a multi-section $\gamma$ of $\pi: X \to \A^1_k$, then the algebraic space $\mathcal{C}$ in the factorisation of $\pi$ in Theorem \ref{factor-A1-fibration} admits a non-constant morphism $\theta \circ \gamma: \A^1_k \to \sC$, so $\sC$ is not $\mathbb{A}^1$-rigid. In Corollary \ref{multi-section exists}, we will prove the converse.
 \end{enumerate}
\end{remark}
\begin{example} \label{multi-section example}
\begin{enumerate}
\item A multi-section can be thought of a section of an $\mathbb{A}^1$-fibration, in a generalized sense. An (algebraic) line bundle has always a section (zero section). Since $\A^1$-fibration is generically trivial, so an $\A^1$-fibration $\pi:X \to C$ has always a rational section, that is there is a rational map $\gamma: C \dashrightarrow X$ such that $\pi \circ \gamma$ is identity (as rational maps). But an $\A^1$-fibration, or even an $\mathbb{A}^1$-bundle does not always admit a section  (\cite[Example 2]{doub}, see also Example \ref{multi-section example}(4)). 
\item Consider, the $\mathbb{A}^1$-fibration $\pi: X \to \mathbb{A}^1_{\mathbb{C}}$ given by the projection to the $x$-axis, where 
$$X \subset \mathbb{A}^3_\mathbb{C} \text{ is given by } x^nz = P(y) - x, \ P(y) \in \mathbb{C}[y] $$
\text{ is a non-constant polynomial} and $n \geq 2$.
Then the map 
$$\gamma: \mathbb{A}^1_{\mathbb{C}} \to X \text{ defined as } t \mapsto (P(t), t, 0)$$
is a multi-section of $\pi$ and moreover $\text{Im}(\gamma)$ intersects every irreducible component of the fiber $\pi^{-1}(0)$ (which is the only degenerate fiber of $\pi$). Thus both $X$ and the algebraic space $\mathcal{C}$, appeared in the factorisation of $\pi$ (by Theorem \ref{factor-A1-fibration}) are $\mathbb{A}^1$-chain connected. 
\item 	A multi-section of an $\A^1$-fibration $\pi: X \to \A^1_\C$ does not always intersect every irreducible components of the degenerate fibers of $\pi$. Consider the following example:

Let $X$ be the surface given by
$$X \equiv\{f(x,y,z)=x^2(x-1)^2z-y(y-1)^2+x(x-1)=0\} \subseteq \mathbb{A}^3_\mathbb{C} $$
and consider the $\A^1$-fibration
$X \xrightarrow{\pr_x}\A^1_\C$ projection to the $x$-axis.
The surface $X$ is smooth, since  
\begin{align*}
    \nabla(f)&=(2x(x-1)(2x-1)z+2x-1,3y^2-4y+1,x^2(x-1)^2) \\
    &\neq (0,0,0);
\end{align*}
as when $x=0,1;\ \text{ then } 2x(x-1)(2x-1)z+2x-1=-1,1$ respectively, otherwise $x^2(x-1)^2\neq 0$. The degenerate fibers are over $x=0,1$, given by $y(y-1)^2=0$. So both the fibers $\pi^{-1}(0)$ and $\pi^{-1}(1)$ have two irreducible components: one irreducible component is reduced, given by $y=0$ and other irreducible component is non-reduced of multiplicity $2$, given by $(y-1)^2=0$. The morphism
$$\gamma(t):\A^1_\C \to X \text{ defined as } t \mapsto (t,t(t-1),t^2-t-2)$$ 
is a multi-section of $\pi$.
Observe that $\text{Im}(\gamma)$ intersects only the reduced components of the degenerate fibers (that is the component given by $y=0$), since $t(t-1) =0$, if $t = 0 \text{ or }1$. Thus the algebraic space $\sC$ (by Theorem \ref{factor-A1-fibration}) is not $\A^1$-invariant. So, $X$ is $\A^1$-connected (from the proof of Theorem \ref{connected-base-new}).
We also know that if $\alpha:\A^1_\C \to X, \ \alpha(t)=(x(t), y(t),z(t))$ is a multi-section , then it cannot intersect the only the non-reduced components of the degenerate fibers, since in that case $x(t),\ (x(t)-1) $ both have to be a square polynomial. \par
	However, we do not know whether $X$ admits a multi-section which intersects every irreducible components of the degenerate fibers. An affirmative answer would imply that $X$ (and also the algebraic space $\sC$) is $\A^1$-chain connected.

\item Consider, a particular example of (2) of an $\mathbb{A}^1$-fibration $\pi: X \to \mathbb{A}^1_k$, the projection to the $x$-coordinate, where $X$
is given by $x^2z = y^2-x$. Then $\pi$ does not admit a section (indeed, if $\gamma \equiv (t, y(t), z(t))$ is a section of $\pi$, then we would have $t^2$ divides $(y(t)^2 - t)$ in $k[t]$, which is not possible), but the morphism $\gamma: \mathbb{A}^1_k \to X$ given by $t \mapsto (t^2,t,0)$ is a multi-section of $\pi$. It is natural to ask when an $\mathbb{A}^1$-fibration admits a multi-section. We will also see an example of an $\mathbb{A}^1$-fibration (Example \ref{rigid example}) over the base $\mathbb{A}^1_k$, which it does not admit even a multi-section.
\end{enumerate}
 
\end{example}

We have the following corollary to Lemma \ref{not-a1-rigid} on the existence of a multi-section of an $\mathbb{A}^1$-fibration over $\mathbb{A}^1_k$.
\begin{corollary} \label{multi-section exists}
Suppose, $\pi:X \to \mathbb{A}^1_k$ is an $\mathbb{A}^1$-fibration from an affine surface $X$ and $X \xrightarrow{\theta} \sC \xrightarrow{\alpha} \A^1_k$ is the factorisation of $\pi$ as in Theorem \ref{factor-A1-fibration}, where $\theta$ is an \'etale $\mathbb{A}^1$-bundle. Then the following are equivalent:
\begin{enumerate}
    \item $\pi$ admits a multi-section.
    \item The algebraic space $\mathcal{C}$ is not $\mathbb{A}^1$-invariant.
    \item There exists a non-constant morphism $\gamma: \A^1_k \to \sC$.
\end{enumerate}
\end{corollary}
\begin{proof}
The equivalence of (2) and (3) follows from Lemma \ref{not-a1-rigid} and (1) implies (2), by
Remark \ref{multi-section remark}). Suppose, $\sC$ is not $\A^1$-invariant, by Lemma \ref{not-a1-rigid} there is a non-constant morphism $\gamma: \mathbb{A}^1_k \to \mathcal{C}$. Consider the pullback diagram
\[\xymatrixcolsep{3pc}
			\xymatrix{ \mathbb{A}^1_k \times_\mathcal{C} X \ar[d]_{\theta^*\gamma} \ar[r]_-{\text{pr}} & X \ar[d]^\theta \\
				\mathbb{A}^1_k \ar[r]^{\gamma } &\sC}
			\]
The morphism $\theta^*\gamma: \mathbb{A}^1_k \times_\mathcal{C} X \to \mathbb{A}^1_k$ is an \'etale $\mathbb{A}^1$-bundle. By the same way in the proof of Theorem \ref{sing-etale-bundle}, we have $H^1_{\text{\'et}}(\mathbb{A}^1_k, \text{Aff}_1) = \{0\}$, so any \'etale $\mathbb{A}^1$-bundle over $\mathbb{A}^1_k$ is trivial. So $\theta^*\gamma$ admits a section, that is there is a morphism $s: \mathbb{A}^1_k \to \mathbb{A}^1_k \times_\mathcal{C} X$ such that $\theta^*\gamma \circ s$ is the identity morphism. Suppose, 
$$\Tilde{\gamma}= \text{pr} \circ s: \mathbb{A}^1_k \to X.$$
Then $\Tilde{\gamma}$ is a lift of $\gamma$, that is $\theta \circ \Tilde{\gamma}  = \gamma$, the lower triangle in the diagram commutes:
\[\xymatrixcolsep{3pc}
			\xymatrix{ \mathbb{A}^1_k \times_\mathcal{C} X \ar[d]^{\theta^*\gamma} \ar[r]_-{\text{pr}} & X \ar[d]^\theta \\
				\mathbb{A}^1_k   \ar@/^/[u]^s \ar[r]_{\gamma} \ar[ur]_{\Tilde{\gamma}} & \sC}
			\]
Since, $\gamma$ is non-constant, so the image of $\Tilde{\gamma}$ is not contained in a fiber of $\theta$. Since $\mathbb{A}^1_k \setminus \{\text{finitely many points}\}$ is $\mathbb{A}^1$-rigid, so the morphism $\pi \circ \Tilde{\gamma}$ is surjective. Thus, $\Tilde{\gamma}$ is a multi-section of $\pi$.
\end{proof}
Suppose, if every fiber of the $\mathbb{A}^1$-fibration $\pi: X \to \mathbb{A}^1_k$ is irreducible.
Then from the construction of $\sC$ (Remark \ref{comments of structure theorem}),
a point of $\mathcal{C}$ corresponds to a point of the base $\mathbb{A}^1_k$, which agains corresponds to a fiber of $\pi$. Thus, if moreover $\mathcal{C}$ is not $\mathbb{A}^1$-invariant, then the non-constant morphism $\gamma: \mathbb{A}^1_k \to \mathcal{C}$ (Lemma \ref{not-a1-rigid}) is surjective on the points and the multi-section $\Tilde{\gamma}: \mathbb{A}^1_k \to X$ (Corollary \ref{multi-section exists}) intersects every fiber of $\pi$. Therefore, in this case both $\mathcal{C}$ and $X$ are $\mathbb{A}^1$-chain connected. We have the following corollary:
\begin{corollary}\label{irred-fibers}
			Suppose, $\pi:X \to \mathbb{A}^1_k$ is an $\mathbb{A}^1$-fibration from a smooth affine surface $X$ such that every fiber of $\pi$ is irreducible.  
Let $\sC$ be the algebraic space in Theorem \ref{factor-A1-fibration}. If $\sC$ is not $\A^1$-invariant, 
			Then both $\sC$ and $X$ are $\A^1$-chain connected. 
\end{corollary}
\begin{remark}
\begin{enumerate}
   \item In Example \ref{rigid example}, we will see that there exists an $\A^1$-fibration $\pi: X \to \A^1_k$ such that every fiber of $\pi$ is irreducible and the algebraic space $\sC$ (by Theorem \ref{factor-A1-fibration}) is $\A^1$-invariant. So $\pi$ does not admit a multi-section and $X$ is also not $\A^1$-connected, here $\pi_0^{\A^1}(X) \cong \sC$.
 \item   Suppose, $\pi: X \to \P^1_k$ is an $\A^1$-fibration from an affine surface $X$ such that every fiber of $\pi$ is irreducible. If moreover, the algebraic space $\sC$ (by Theorem \ref{factor-A1-fibration}) is not $\A^1$-invariant, then by Lemma \ref{not-a1-rigid}, there is a non-constant morphism $\gamma: \A^1_k \to \sC$, which can miss at most one point of $\sC$ (since $\sC$ is the algebraic space over $\P^1$) and consequently, the morphism $\Tilde{\gamma}: \A^1_k \to X$ (from Lemma \ref{not-a1-rigid}) may not intersect at most one fiber of $\pi$. Thus, we don't know whether $X$ or $\sC$ is $\A^1$-chain connected in this case. But $\sC$ (hence $X$ is also) is $\A^1$-connected (Theorem \ref{connected-base-new}).
 \end{enumerate}
\end{remark}
\begin{example}
Consider again Example \ref{multi-section example}(2) with $P(y) = (y-\alpha)^m$, for some $\alpha \in k$  fixed \cite[Example 5]{doub}
 $$X \equiv x^nz=(y-\alpha)^m-x \text{ and } \pi: X \to \A^1_k \text{ defined as }(x,y,z)\mapsto x.$$
Then, every fiber of $\pi$ is irreducible and $\pi$ admits a multi-section. So the algebraic space $\sC$ in Theorem \ref{factor-A1-fibration} is not $\A^1$-invariant. Thus, both $X$ and $\sC$ are $\A^1$-chain connected.
\end{example}

If every fiber of the $\A^1$-fibration $\pi; X \to C$, where $C \cong \A^1_k \text{ or } \P^1_k$, is reduced, then the algebraic space $\sC$ (in Theorem \ref{factor-A1-fibration}) is a scheme (possibly non-separated), 
isomorphic to $C$ with multiple points. Therefore, $\sC$ is $\A^1$-chain connected and for each copies of $\A^1$ in $\sC$ would lift to an $\A^1$ in $X$ (since $\theta: X \to \sC$ is \'etale locally trivial $\A^1$-bundle, by Theorem \ref{factor-A1-fibration}). Thus, $X$ is also $\A^1$-chain connected (see Section \ref{section a1 homotopy types} for the $\A^1$-homotopy types of such surfaces). Hence we have the following corollary:
\begin{corollary}
    Suppose, $\pi:X \to C$, where $C \cong \A^1_k \text{ or } \P^1_k$, is an $\mathbb{A}^1$-fibration from a smooth affine surface $X$ such that every fiber of $\pi$ is reduced and $\sC$ is the algebraic space in Theorem \ref{factor-A1-fibration}. Then both $\sC$ and $X$ are $\A^1$-chain connected.
\end{corollary}
More generally, 
we have the following result:

			\begin{theorem}\label{connected-base-new}
            Let $X$ be a smooth affine surface, which admits an $\A^1$-fibration
$\pi: X \rightarrow C$ is an onto a smooth, rational curve $C$. Then $\pi_0^{\A^1}(X)$ is $\A^1$-invariant.
			\end{theorem}
		\begin{proof}
			Following the notations in Theorem \ref{factor-A1-fibration} and Construction \ref{construction of algebraic space}
            , the $\mathbb{A}^1$-fibration $\pi:X \to C$ factorises as:
			$$X \xrightarrow{\theta} \mathcal{C} \xrightarrow{p} \check{C} \xrightarrow{\check{p}}C.$$ 
			In this factorisation $\theta:X\to\sC$ is a \'etale $\A^1$-bundle over the algebraic space $\sC$, which is not a scheme if $\pi$ has a non-reduced fiber. Each point of $\check{C}$ corresponds to an irreducible component of a fiber of $\pi$. 
            If every fiber of $\pi$ is reduced, then $\sC\xrightarrow{p}\check{C}$ is an isomorphism.
			By Corollary \ref{factorisation A1 equivalence}, the morphism $\theta: X \to \mathcal{C}$ is an $\mathbb{A}^1$-weak equivalence; therefore 
			$$\pi_0^{\mathbb{A}^1}(X) \xrightarrow{\cong} \pi_0^{\mathbb{A}^1}(\mathcal{C}).$$
			Thus if $\sC$ is $\A^1$-invariant, then 
			$$\pi_0^{\mathbb{A}^1}(X) \cong \pi_0^{\A^1}(\sC) \cong \sC \text{ is also } \A^1 \text{-invariant}.$$ 
			For the remaining case, assume that $\sC$ is not $\A^1$-invariant. In this case, we will prove that $\mathcal{C}$ is $\mathbb{A}^1$-connected. This would imply that $X$ is also $\mathbb{A}^1$-connected.

To prove that $\mathcal{C}$ is $\mathbb{A}^1$-connected, by \cite[Lemma 3.3.6]{ictp} we need to prove that $\pi_0^{\mathbb{A}^1}(\mathcal{C})(Spec \ F)$ is trivial, for every finitely generated field extension $F/k$. The algebraic space $\mathcal{C}$ of dimension $1$ has only two types of points: the generic point, say $\eta$ having residue field $k(t)$ (the morphism $\alpha = \check{p} \circ p: \sC \to C$ is birational and $C$ is a rational curve) and the $k$-points. Since we have assumed that $\sC$ is not $\A^1$-invariant, so by Lemma \ref{not-a1-rigid}, 
there is a non-constant morphism $\gamma: \mathbb{A}^1_k \to \mathcal{C}$. Since $\sC$ is of dimension $1$, so $\text{Im}(\gamma)$ contains the generic point $\eta$. Thus, for any $x \in \sC(k)$ such that $x \in \text{Im}(\gamma)$, we have 
$$\eta = x \in \sS(\sC)(k(t)).$$
Hence to prove $\sC$ is $\A^1$-connected, we only need to show $\pi_0^{\A^1}(\sC)(k)$ is trivial. We will actually prove that $\sS^2(\sC)(k)$ is trivial, that is any two points 
$\alpha,\beta\in\sC(k)$ we will prove that $\alpha =\beta\in\sS^2(\sC)(k)$. If both the points $\alpha, \beta \in \text{Im}(\gamma)$, then $\alpha = \beta \in \sS(\sC)(k)$. So without loss of generality, we can assume that $\alpha \in \mathrm{Im}(\gamma)$ and $\beta \notin \mathrm{Im}(\gamma)$. Suppose, $p(\beta) = \check{\beta} \in \check{C}$.

Since $\gamma$ is non-constant and $\sC$ is of dimension $1$, so $\text{Im}(\gamma)$ misses only finitely many $k$-points of $\sC$.

			Following construction of the algebraic space $\sC$ as in \cite[Theorem 9]{doub}, 
            there is a germ of smooth curve $\tilde{C}_\beta\hookrightarrow X$ intersecting the fiber $(p \circ \theta)^{-1}(\check{\beta})$
            transversally. Since $X$ is a rational variety, therefore we can choose $\tilde{C}_\beta$ to be a rational curve \cite[Proposition 3.4]{bhs1}. 
            By shrinking $\tilde{C}_\beta$ if necessary, we can assume the following:
			\begin{enumerate}
			\item There exist a unique point $\tilde{\beta}\in \tilde{C}_\beta$ such that $\theta(\tilde{\beta}) = \beta.$ 
			\item The morphism $\phi:= \tilde{C}_\beta \xrightarrow{\theta} \sC \xrightarrow{p} \check{C}$
            is quasi finite onto its image.
			\item Im$(\phi)$ is an affine, open neighbourhood of $\check{\beta}$ in $\check{C}$.
			\item For a point $y \in \text{Im}(\phi)$, if $y \neq \check{\beta}$, then the fiber $(p \circ \theta)^{-1}(y)$ is reduced and irreducible.
			\item $U:=\theta(\tilde{C}_\beta) \setminus \{\beta\} \subseteq \text{Im}(\gamma)$, since $\text{Im}(\gamma)$ misses only finitely many $k$-points of $\sC$.
            \item The \'etale $\A^1$-bundle $\theta: \theta^{-1}(U) \to U$ is the trivial $\A^1$-bundle over $U$, since the $\A^1$-fibration $\pi$ is generically trivial.
			\end{enumerate}

 By (5), $\theta(\tilde{C}_\beta \setminus \{\tilde{\beta}\}) \subseteq \text{Im}(\gamma)$. We claim that there is a morphism 
 $$\tilde{\theta}: \tilde{C}_\beta \setminus \{\tilde{\beta}\} \to \A^1_k$$
 such that the restriction $\theta: \tilde{C}_\beta \setminus \{\tilde{\beta}\} \to \sC$ factors through $\tilde{\theta}$, that is
 $$\theta|_{\tilde{C}_\beta \setminus \{\tilde{\beta}\}} = \gamma \circ \tilde{\theta}.$$
Indeed following the proof of Lemma \ref{not-a1-rigid}, $\gamma$ is constructed as follows
$$\gamma= \theta \circ \tilde{\gamma}, \  \A^1_k\xrightarrow{\tilde{\gamma}} X \xrightarrow{\theta} \sC ,$$
such that $\tilde{\gamma}|_{\gamma^{-1}(U)}$ is the following composition of morphisms
$$\gamma^{-1}(U)\xrightarrow{s_0} \gamma^{-1}(U)\times_\sC X \to U \times_\sC X\xhookrightarrow{\text{open}} X,$$
where $s_0$ is the zero section of the morphism $\gamma^{-1}(U)\times_\sC X \to \gamma^{-1}(U)$ (denote it by $\theta_\gamma$), which is the trivial $\A^1$-bundle.
Consider the following commutative diagram. 
 	\[
 	\xymatrix{
 		\gamma^{-1}(U)\times_\sC X \ar[d]^{{\theta}_\gamma} \ar[r] & U\times_\sC X \ar[d]^{{\theta}_U} \ar[r] & X \ar[d]^{\theta} &  \\
 		\gamma^{-1}(U) \ar@/^/[u]^{s_0} \ar@{->>}[r] & U \ar@/^/[u]^{s_0} \ar@{^{(}->}[r] & \sC & \A^1_k \ar[ul]_{\tilde{\gamma}} \ar[l]^\gamma  }     
 	\]
    Here both the morphisms $\theta_\gamma:\gamma^{-1}(U)\times_\sC X \to \gamma^{-1}(U)$ and $\theta_U: U \times_\sC X \to U$ are the trivial $\A^1$-bundles and their corresponding zero sections are denoted by $s_0$ (by shrinking $\tilde{C}_\beta$, we can assume $\gamma^{-1}(U) \to U$ is surjective).
    
 From the diagram, we have Im$(\tilde{q}:U\xrightarrow{s_0}U\times_\sC X \to X)\subset \mathrm{Im}(\tilde{\gamma}) $.
 Since $X$ is affine, so Im$(\tilde{\gamma})$ is closed in $X$ and the normalisation of $\text{Im}(\tilde{\gamma})$ is $\A^1_k$. Therefore $\tilde{q}$	factors through the normalisation map $\A^1_k \to \text{Im}(\tilde{\gamma})$. Thus, there exists a morphism $q^\prime: U \to \A^1_k$ such that $\tilde{q}: U \to X$ is the composition of the morphisms
 $$U\xrightarrow{q'}\A^1_k \xrightarrow{\text{normalisation}}\mathrm{Im}(\tilde{\gamma}) \hookrightarrow X.$$
 Consequently, 
 $$\theta|_{\tilde{C}_\beta \setminus \{\tilde{\beta}\}} = \gamma \circ \tilde{\theta}, \text{ where } \tilde{\theta}:=q^\prime \circ \theta|_{\tilde{C}_\beta \setminus \{\tilde{\beta}\}}.$$
 As $\tilde{C}_\beta$ is a rational curve, we choose an open immersion $j:\tilde{C}_\beta \hookrightarrow \A^1_k, $ such that $\tilde{\beta}\mapsto 0$. Let us denote the canonical immersion $\A^1_k\backslash\{0\}\hookrightarrow\A^1_k$ by $i$. 

 Now the $\A^1$-ghost homotopy between $\alpha$ and $\beta$ is as follows:
 \begin{equation}
 	\label{new-main-ghost-homotopy}
 	\xymatrixcolsep{5pc}
 	\xymatrix{V \times_{\mathbb{A}^1_k} V \ar@<-.5ex>[r]_-{\pr_2} \ar@<.5ex>[r]^-{\pr_1} &{V:=\tilde{C}_{\beta}\sqcup (\A^1_k\backslash\{0\}}) \ar[dd]^{(j,i)} \ar[r]^-{H=\theta|_{\tilde{C}_\beta } \sqcup \alpha}&\sC \\
 		\\
 		{\mathrm{Spec}(k)} \ar@<-.5ex>[r] \ar@<.5ex>[r] \ar@<-.5ex>[uur]_{i_1} \ar@<.5ex>[uur]^{\tilde{\beta}} &{\A^1_k}}
 \end{equation}

In the diagram \ref{new-main-ghost-homotopy}, $(j,i): V \to \A^1_k$ is 
is a Nisnevich cover of $\A^1_k$, where
$$V:=\tilde{C}_{\beta} \sqcup (\A^1_k \backslash \{0\}),  $$
along with the open immersions $j$ and $i$ (so it is actually a Zariski covering).
Suppose, $H : V \to \mathcal{C}$ is given by the restriction $\theta: \tilde{C}_\beta \to \sC$ and
and $H|_{\mathbb{A}^1_k \setminus \{0\}}$ is the constant morphism to $\alpha$. The Zariski covering $V \to \mathbb{A}^1_k$ has the lifts $\Tilde{\beta} \in \Tilde{C}_\beta$ and $1 \in \mathbb{A}^1_k \setminus \{0\}$ of $0$ and $1$ in $\mathbb{A}^1_k$ respectively. And $H$ maps $\Tilde{\beta}$ to $\beta$ and $1$ to $\alpha$.
Now,
$$V \times_{\mathbb{A}^1_k} V = \Tilde{C}_\beta \sqcup (\mathbb{A}^1_k \setminus\{0\}) \sqcup (\Tilde{C}_\beta \times_{\mathbb{A}^1_k} \mathbb{A}^1_k \setminus\{0\}) \sqcup (\mathbb{A}^1_k \setminus\{0\} \times_{\mathbb{A}^1_k} \Tilde{C}_\beta)$$
and there are two morphisms
$$\text{pr}_1, \text{pr}_2: V \times_{\mathbb{A}^1_k} V \to V.$$
To give an $\mathbb{A}^1$-ghost homotopy between $\alpha$ and $\beta$ (that would be $H$) we need to give a naive $\mathbb{A}^1$-homotopy between 
$$H \circ \text{pr}_1 \text{ and } H \circ \text{pr}_2.$$
Now, the morphisms $\text{pr}_1$ and $\text{pr}_2$ are same to the connected components $\Tilde{C}_\beta$ and $\mathbb{A}^1_k \setminus\{0\}$. The other two connected components are both isomorphic to $\tilde{C}_{\beta}\backslash\{\tilde{\beta}\}$. But $H \circ \text{pr}_1$ factors as
$$\tilde{C}_{\beta}\backslash\{\tilde{\beta}\} \hookrightarrow \Tilde{C}_\beta \xrightarrow{\theta} \mathcal{C}$$
and $H \circ \text{pr}_2$ factors as
$$\tilde{C}_{\beta}\backslash\{\tilde{\beta}\} \xrightarrow{j} \mathbb{A}^1_k \setminus \{0\} \xrightarrow{\alpha} \mathcal{C}.$$
But we have proved that $\theta|_{\tilde{C}_\beta \setminus \{\tilde{\beta}\}}$ factors as
$$\tilde{C}_\beta \setminus \{\tilde{\beta}\} \xrightarrow{\tilde{\theta}} \A^1_k \xrightarrow{\gamma} \sC.$$
So there the $\mathbb{A}^1$-homotopy
$$G:\A^1_{\tilde{C}_\beta \setminus \{\tilde{\beta}\}} \to \sC$$
defined as $G = t (H \circ \text{pr}_1) + (1-t) (H \circ \text{pr}_2)$. Thus $H$ defines an $\A^1$-ghost homotopy between $\alpha$ and $\beta$. So, $\alpha = \beta \in \sS^2(\sC)(k)$.




 Hence $\sS^2(\sC)(k )$ is trivial. Since $\eta = x\in \sS(\sC)(k)$, for some $k$-point $x \in \text{Im}(\gamma)$, so for every finitely generated field extension $F/k$, $\sS^2(\sC)(F)$ is trivial and thus, $\pi_0^{\A^1}(\sC)(F)$ is trivial. Hence $\sC$ is $\A^1$-connected, by \cite[Lemma 3.3.6]{ictp}. Therefore, by Corollary \ref{factorisation A1 equivalence}, $X$ is also $\A^1$-connected. 
 \end{proof}
 	\begin{remark}
 	    Thus if $X$ is a rational, affine surface admitting an $\A^1$-fibration $\pi: X \to C$ 
        onto a smooth curve $C$ (the curve $C$ must be a rational curve) and $\sC$ is the algebraic space (by Theorem \ref{factor-A1-fibration}), then by Corollary \ref{a1 invariance of algebraic space} and Theorem \ref{connected-base-new} we have the following:
        \begin{enumerate}
            \item If $C$ is a $\A^1$-rigid curve, then $\pi_0^{\A^1}(X) \cong \sC$.
            \item If $C \cong \A^1_k \text{ or } \P^1_k$, then 
            \begin{enumerate}
                \item either $\sC$ is $\A^1$-invariant, in that case $\pi_0^{\A^1}(X) \cong \sC$ (Example \ref{rigid example}).
                \item or $\sC$ is not $\A^1$-invariant. In this case $\sS^2(\sC)(F)$ is trivial, for every finitely generated separable field extension $F/k$. So both $\sC$ and $X$ are $\A^1$-connected.
            \end{enumerate}
        \end{enumerate}
 	\end{remark}
\begin{corollary} \label{negative kodaira invariant}
Let $X$ be a smooth, affine surface over an algebraically closed field of characteristic zero. Suppose, $X$ is $\A^1$-uniruled. Then $\pi_0^{\A^1}(X)$ is $\A^1$-invariant.
\end{corollary}
Therefore by Corollary \ref{rigid case a1 invariant} and Corollary \ref{negative kodaira invariant}, we have the $\A^1$-invariance of the $\pi_0^{\A^1}(X)$:
\begin{corollary}\label{main corollary}
Let $X$ be a smooth, affine surface over an algebraically closed field of characteristic zero. 
Then $\pi_0^{\A^1}(X)$ is $\A^1$-invariant.
\end{corollary}

\begin{corollary} \label{naive spaces}
Suppose, $X \in Sm/k$ is an affine surface, which is $\A^1$-uniruled, admitting an $\A^1$-fibration $\pi: X \to C$. Also, assume that the algebraic space $\sC$ in Theorem \ref{factor-A1-fibration} is $\A^1$-invariant (for example, if $C$ is an $\A^1$-rigid curve, by Corollary \ref{a1 invariance of algebraic space}). Then $X$ is $\A^1$-naive \cite[Definition 2.1.1]{asokII}, in particular $Sing_*(X)$ is $\A^1$-local.
\end{corollary}
			\begin{example} \label{rigid example}
 There are $\mathbb{A}^1$-fibrations $X \to \mathbb{A}^1$ such that $\sC$ is $\mathbb{A}^1$-rigid. Consider the smooth affine surface $X\equiv\{x^2(x-1)^2z-y^2+x(x-1)=0\}$ and $\pr_x:X\to \A^1$ is the fibration. For any morphism $\gamma: \A^1 \to X$, $\pr_x\circ\gamma$ is constant. Clearly if $\gamma(t)=(x(t),y(t),z(t))$ then 
	\begin{align*}
		x(t)^2(x(t)-1)^2z(t)-y(t)^2+x(t)(x(t)-1)&=0\\
		x(t)(x(t)-1)[x(t)(x(t)-1)z(t)+1]&=y(t)^2
	\end{align*}
	Thus any root $t_0$ of $x(t)\mathrm{\ or\ }(x(t)-1)$ must be a root of $y(t)^2$. So the roots of $x(t)$ and $(x(t)-1)$ must have multiplicity of multiple of $2$ since $x(t_0)(x(t_0)-1)z(t_0)+1=1$ for any root $t_0$ of $x(t)\mathrm{\ or\ }(x(t)-1)$. Thus both of them are square polynomial. Let $x(t)=u(t)^2,\ x(t)-1=v(t)^2$, then $$1=u(t)^2-v(t)^2=(u(t)-v(t))(u(t)+v(t)).$$ So both $(u(t)+v(t))$ and $(u(t)-v(t))$ are constant then both $u(t),v(t)$ are constant, so $\pr_x\circ\gamma(t)=x(t)=u(t)^2$ is constant.\par Now, let $\sC$ is the algebraic space in the factorisation $X\xrightarrow{\theta}\sC\xrightarrow{p}\A^1$. Then by Lemma \ref{irred-fibers} any non-constant map $\gamma:\A^1_k\to \sC$ is surjective. Also as $X\to \sC$ is a étale bundle any non-constant map $\gamma: \A^1_k \to \sC$ lifts to $X$ (Corollary \ref{not-a1-rigid}) Which must lie in a single fiber. Thus there does not exist any non-constant map $\gamma:\A^1_k\to \sC$ so $\sC$ is $\A^1$-rigid.  
%
%

			\end{example}


			\begin{remark}
 From Theorem \ref{connected-base-new} we see that if the algebraic space $\sC$ appearing in Theorem \ref{factor-A1-fibration} is not $\A^1$-invariant, then it is $\A^1$-connected by proving $\sS^2(\sC)(k)$ is trivial. 
 However we do not know whether in those cases $\sS(\sC)(k)$ is trivial (equivalently, $X$ is $\A^1$-chain connected) or not. Note that, for such an example of an $\A^1$-fibration $\pi:X \to C$ with $\sS(\sC)(k)$ non-trivial (and $\sS^2(\sC)(k)$ is trivial), the cannonical epimorphism $\sS(X)\to \pi_0^{\A^1}(X)$ is not an isomorphism, thus $Sing_*(X)$ is not $\A^1$-local (see \cite[Conjecture 2.2.8]{am}).
\end{remark} 
\section{$\mathbb{A}^1$-homotopy type of $\mathbb{A}^1$-ruled surfaces} \label{section a1 homotopy types}
In the previous section we have seen that an $\A^1$-fibration $\pi : X \to C$ where $X$ is a smooth affine surface and $C$ is a smooth curve factors as $X \xrightarrow{\theta}\sC\xrightarrow{p} C$ where $\theta$ is an étale $\A^1$-bundle, $\sC$ is a smooth algebraic space and $p$ is birational and quasi finite. We have also deduced that $\theta$ is an $\A^1$-weak equivalence.
 In this section, we describe the $\A^1$-homotopy types of an $\A^1$-uniruled surfaces in the following cases: 
the base curve $C$ of the $\A^1$-fibration $\pi: X \to C$ is $\A^1$-rigid
or every fiber of the $\A^1$-fibration $\pi: X \to C$ is reduced (in this case, the algebraic space $\sC$ is actually a scheme).


\subsection{The base curve $C$ is an $\mathbb{A}^1$-rigid curve}
Suppose, $X$ admits an $\mathbb{A}^1$-fibration $\pi: X \to C$ to an $\mathbb{A}^1$-rigid curve $C$. Then by Theorem \ref{factor-A1-fibration}, $\pi$ factors as
$$X \xrightarrow{\rho} \mathcal{C} \xrightarrow{p} C,$$
where $\rho$ is an étale locally trivial $\mathbb{A}^1$-bundle. By Corollary \ref{factorisation A1 equivalence}, $\rho$ is an $\mathbb{A}^1$-weak equivalence. Since $C$ is $\mathbb{A}^1$-rigid, so the algebraic space $\mathcal{C}$ is an $\mathbb{A}^1$-rigid Nisnevich sheaf (an algebraic space is an étale sheaf, so is a Nisnevich sheaf). Therefore, the $\mathbb{A}^1$-homotopy types of $\mathcal{C}$ is equivalent to the isomorphism types of $\mathcal{C}$ (Corollary \ref{homotopy types base rigid}). Thus the class of $\mathbb{A}^1$-homotopy types of a surfaces admitting an $\mathbb{A}^1$-fibration to an $\mathbb{A}^1$-rigid curve consists all the smooth one dimensional algebraic spaces (schemes or non-separated) that can be constructed in the factorisation. In particular this class consists all $\mathbb{A}^1$-rigid curves $C$ with multiple points (in the case when every fiber of $\pi$ is reduced), alongwith the non-separated algebraic spaces obtained from $C$, if there is a non-redued fiber.

		\subsection{The base curve $C$ is $\mathbb{A}^1$-connected}

In this subsection, we will find the $\mathbb{A}^1$-homotopy types of the smooth $\mathbb{A}^1$-ruled surfaces $X$ that admit $\mathbb{A}^1$-fibration $\pi: X \to C$ to an $\mathbb{A}^1$-connected curve with  every fiber of $\pi$ is reduced. 



Since $C$ is $\mathbb{A}^1$-connected, therefore $C$ is isomorphic to $\mathbb{A}^1$ or $\P^1$.
Let us recall the purity isomorphism in $\mathcal{H}_\bullet(k)$ by Morel and Voevodsky. 
\begin{theorem} \cite[Theorem 2.23, Section 3.2]{mv} \label{homotopy purity}
	Let $i: Z \to X$ be a closed immersion of schemes in $Sm/k$ and $N_{X,Z} \to Z$ is the normal bundle over $Z$ of rank $n=\text{codim}_Z(X)$. Then there is an $\mathbb{A}^1$-weak equivalence of pointed sheaves
	$$X/(X \setminus i(Z)) \sim Th(N_{X, Z}),$$
	where for a vector bundle $\rho: V \to X$, the Thom space of $\rho$ is defined as
	$$Th(\rho):= V/(V \setminus s(X)),$$
	here $s: X \to V$ is the zero section of $\rho$.
\end{theorem}

		\begin{example}
\begin{enumerate}
	\item Let $X=\tilde{\A}^1$ denotes $\A^1$ with \textit{double} origins. It is the homotopy pushout of the following diagram

	\[
	\xymatrix{{\G_m} \ar@{^{(}->}[d] \ar[r] & {\A^1\cong*} \ar[d] \\
		{\A^1} \ar[r] & {\tilde{\A}^1\cong\A^1/\G_m\cong\P^1.}}
	\]
	As $\A^1/\G_m$ is the homotopy pushout since $\G_m\hookrightarrow\A^1$ is cofibration. Then $\A^1/\G^m$ is the Thom space of the $1$-dimetional vector bundle over a point. Hence, by Theorem \ref{homotopy purity} it is isomorphic to $\P^1$. 
		\item Let $X=\A^1_{0;n}$ denotes $\A^1$ with $n$ origins, formed by gluing $n$ copies of $\A^1$ along $\A^1\backslash\{0\}$ via identity. Also denote the origins by $0_1,\dots,0_n$. 
	By the homotopy pushout diagram we see, 
	
	\[
	\xymatrix{{\G_m} \ar@{^{(}->}[d] \ar[r] & {\A^1\cong*} \ar[d] \\
		{\A^1_{(n-1);0}} \ar[r] & {\A^1_{n;0}\cong (\A^1_{(n-1);0}/\G_m). }
	}
	\]
	Let $Z=\{0_1,\dots,0_{(n-1)}\}\hookrightarrow\A^1_{(n-1);0}$ be the smooth closed embedding. Then by homotopy purity (Theorem \ref{homotopy purity}) we see, $$(\A^1_{(n-1);0}/\G_m)=Th(N_{Z,X})=\Sigma_T^1(Z)_+=\P^1\wedge(Z)_+=(\P^1)^{\vee (n-1)}$$ as $Z$ is the disjoint union of $(n-1)$ points, thus the normal bundle is trivial. Hence, the Thom space is wedge of $(n-1)$-many copies of $\P^1$. The calculations are done as it was mentioned in \cite[Proposition 5.3.4.3]{as}\\
	
	The Picard group of $\A^1_{0;n}$ is $\Pic(\A^1_{0;n})=\Z^{(n-1)}$, direct sum of $(n-1)$ copies of $\Z$. Indeed, using Weil divisors we see, $nc = \Div((x-c)^n)=0$ in $\Pic(\A^1_{0;n})$ if $c$ is not one of the origins. Let $q:\A^1_{0;n}\to \A^1$ by $0_i\mapsto 0.\ c\mapsto c$, then $\Div(q)=\Sigma_{1}^{n}(0_i)=0$ in $\Pic(\A^1_{0;n})$. Let $f:\A^1_{0;n}=\tilde{\A}^1 \to\A^1$ be any map, then $f(0_i)=f(0_j)$ for all $i,j$.  Let $U=\Spec(k(x)),\ T=\Spec(k[x]_{(x)})$

	\[\xymatrixcolsep{4pc}
	\xymatrix{{U=\Spec(k(x))} \ar@{^{(}->}[d] \ar[r]^<<<<<<<<<\eta & {\tilde{\A}^1} \ar[r]^f & {\A^1} \\
		{T=\Spec(k[x]_{(x)})} \ar@<-1.3px>[ur]_{s_j} \ar@<1.3px>[ur]^{s_i}
	}
	\]
	where $\eta$ maps to the generic point of $\tilde{\A}^1$. Then the two extension $s_i,\ s_j$ mapping the closed point to $0_i\mathrm{\ and \ }0_j$ respectively must be a unique map after composing with $f$. Thus $f(0_i)=f(0_j)$ for all $i,j$. Thus, we have $$\Pic(\A^1_{0;n})=\Z[0_1,\dots0_n]/\sum_{1}^{n}(0_i)\cong \Z^{(n-1)}.$$ Hence $\A^1_{0;n}\mathrm{\ are\ }\A^1_{0;m}$ is not $\A^1$-homotopic if $m\neq n$.

	\item Let $X=\A^1_{0,1;2,2}\ i.e.,\ \A^1$ with \textit{double} origin and \textit{double} 1. 
	\[
	\xymatrix{{\G_m} \ar@{^{(}->}[d] \ar[r]^{t\mapsto(t+1)} & {\A^1\cong*} \ar[d] \\
		{\A^1_{2;0}} \ar[r] & {\A^1_{0,1;2,2}\cong\A^1_{0;2}/\G_m}
	}
	\]
	By the above homotopy pushout diagram we see that it is $\A^1$-homotopic to $\A^1_{0;2}/\G_m$, which is also $\A^1$-homotpic to $\A^1_{0;3}\ i.e.,\ \A^1$ with $3$ origins.\\
	
	However they are not isomorphic as schemes. By similar argument as in the last example we have for any map $f:X=\A^1_{0,1;2,2}\to \A^1$ $f(0_i)=f(1_j)$ for all $1\leq i,j\leq 2$. Now, for any $g:X\to \A^1_{(0,3)}$ we have the composition $$X\xrightarrow{g}\A^1_{(0,3)}\xrightarrow{q}\A^1.$$ Here $q$ maps all of the origins to $0$ of $\A^1$ and identity elsewhere. Thus $g(0_i),g(1_j)\subset \{0_1,0_2,0_3\}$ since $q\circ g(0_i)=q\circ g(1_j)$ but $q(x)=q(y)$ if $x=y \mathrm{\ or\ } x,y \in \{0_1,0_2,0_3\}.$ Hence, $g$ cannot be injective.

	\item From the previous case for $n\geq2$, $X=\P^1_{0;(n-1)}$ i.e., $\P^1$ with $(n-1)$ origins is isomorphic to $\A^1_{0;n}$. The case $n=2$  is done in the first example. Using induction it follows from the same homotpy purity diagram 
	
	\[
	\xymatrix{{\G_m} \ar@{^{(}->}[d] \ar[rr]&& {\A^1\cong*} \ar[d] \\
		{\P^1_{0;(n-2)}\cong\A^1_{0;(n-1)}} \ar[rr] && {\P^1_{0;(n-1)}\cong\A^1_{0;(n-1)}/\G_m\cong\A^1_{0;n}.}
	}
	\]
\end{enumerate}
		\end{example}

		Now we generalise this computations in the following lemma.
		\begin{lemma}
			Suppose, $C_1$ be the curve affine line with $m$ origins and Suppose there are $r$ many $k$-points $a_1, \dots, a_r \in \mathbb{A}^1$ along with the integers $n_1, \dots, n_r$ such that $\sum_{1}^{r}(n_i-1)=m-1$. $C_2$ be the affine line with the points $a_i,\ n_i$ many times denoted as,
			$C_2:= \mathbb{A}^1(a_1, n_1; \dots; a_r, n_r)$ 
			Then $C_1$ and $C_2$ are $\A^1$-weakly equivalent.
			
		\end{lemma}
		\begin{proof}
			 Let $C_2= \mathbb{A}^1(a_1, n_1; \dots, a_r, n_r)$. Without loss of generality we can assume $a_1=0,\ a_2=1$. By induction, also assume $\mathbb{A}^1(a_1, n_1; \dots, a_r, n_r-1) \cong \A^1_{0;1+(\sum_{0}^{(r-1)}(n_i-1))+(n_r-2)}$ the affine line with $(1+(\sum_{0}^{(r-1)}(n_i-1))+(n_r-2))$ many origins. From the pushout diagram we have $C_2$ is the homotopy pushout of the following diagram

			\[
			\xymatrix{{\G_m} \ar@{^{(}->}[d] \ar[rr]^{t\mapsto(a_r+t)}&& {\A^1 \cong *}\ar[d] \\
				{\A^1_{0;1+ (\sum_0^{r-1}(n_i-1))+(n_r-2)}\cong \mathbb{A}^1(a_1, n_1; \dots, a_r, n_r-1)} \ar[rr]&& X.
			}
			\]
			Thus $C_2\cong C_1$. Also by similar argument we can conclude that any morphism between $C_2\to C_1$ cannot be injective, thus they are not isomorphic.
		\end{proof}
		\begin{remark}
			We see that for $\A^1$-connected smooth curve, their $\A^1$-homotopic type is not same as their isomorphic type, however for $\A^1$-rigid curves it is not the case. If $C$ is an $\A^1$-rigid curve (\textit{not-necessarily separated}) then $\pi_0^{\A^1}(C)=C$ as a sheaf, thus in this case isomorphic type is same as $\A^1$-homotpy type. 
		\end{remark}

\section{Appendix} \label{detailed horn formula}

\begin{point}[Boundary and degeneracy maps of $Sing_*(\sX)(U)$]\label{maps of sing}~
	
Let $\sX$ be a space and $U$ be a scheme over $k$. Here we describe the boundary and degeneracy maps of the simplicial set $Sing_*(\sX)(U)$. They are denoted by $d_i$ and $s_i$ respectively. 
$$d_i: Sing_*(\mathcal{X})(U)_n \to Sing_*(\mathcal{X})(U)_{n-1}, \ 0 \leq i \leq n$$
are induced by the maps $d^i: \Delta^{n-1}_a \to \Delta^n_a$, which are given by the morphism between $k$-algebras (we again denote it by $d^i$)
$$d^i:  \frac{k[x_0, x_1,\dots, x_n]}{(\sum_{i=0}^n x_i - 1)} \to  \frac{k[x_0, x_1,\dots, x_{n-1}]}{(\sum_{i=0}^{n-1} x_i - 1)}$$
defined as 
\[
\begin{cases}
	\overline{x_j} \mapsto \overline{x_j} & j < i \\
	\overline{x_j} \mapsto \overline{0} & j = i \\
	\overline{x_j} \mapsto \overline{x_{j-1}} & j > i
\end{cases}
\]
(here, for $f \in k[x_0,\dots, x_n]$ $\bar{f}$ denotes the class in $\frac{k[x_0, x_1,\dots, x_n]}{(\sum_{i=0}^n x_i - 1)}$). 
Similarly, here the degeneracy maps 
$$s_i: Sing_*(\mathcal{X})(U)_n \to Sing_*(\mathcal{X})(U)_{n+1}, \ 0 \leq i \leq n$$
are induced by the maps $s^i: \Delta^{n+1}_a \to \Delta^n_a$, 
which are given by the morphism between $k$-algebras (we again denote it by $s^i$)
$$s^i:  \frac{k[x_0, x_1,\dots, x_n]}{(\sum_{i=0}^n x_i - 1)} \to  \frac{k[x_0, x_1,\dots, x_{n+1}]}{(\sum_{i=0}^{n+1} x_i - 1)}$$
defined as 
\[
\begin{cases}
	\overline{x_j} \mapsto \overline{x_j} & j < i \\
	\overline{x_i} \mapsto \overline{x_i + x_{i+1}} & j = i \\
	\overline{x_j} \mapsto \overline{x_{j+1}} & j > i
\end{cases}
\]
\end{point}
To give the horn filling formula for $Sing_*(\A^m_k)(U)$ for $U \in Sm/k$ first we recall the horn filling formula for simplicial groups.

\begin{lemma}\cite[Lemma 3.1]{curtis-simplicial}\label{horn-formula}
Let $G$ be a simplicial group. $(f_1,\dots,f_{l-1},f_{l+1},f_n) \in G_{n-1} \mathrm{\ such\ that\ }d_i(f_j)=d_{j-1}(f_i)\ i<j\neq k;\ $ be a horn. Then take -
	\begin{align*}
		g_0 &= s_0(f_0)& \\
		g_i &= g_{i-1}\cdot s_i(d_i(g_{i-1}))^{-1}\cdot s_i(f_i)\  &\mathrm{ for}\ 0<i<l\\
		g_i &= g_{i+1} \cdot s_{i-1}(d_i(g_{i+1}))^{-1}\cdot s_i(f_i)\  & \mathrm{for}\ l<i<n\\
		g_n &= g_{l-1}\cdot s_{n-1}(d_n(g_{l-1}))^{-1}\cdot s_{n-1}(f_n).&
	\end{align*}
	Where $\cdot$ is the group operation. When $k=0$ we set $g_{-1}=e_G$ the identity element. Then $g_l\in G_n$ satisfies the property $d_j(g_{l+1})=f_j$ for $0\leq j \leq n$.
\end{lemma}

Clearly for $U\in Sm/k\ Sing_*(\A^m_K)(U)$ is a simplicial group. The group operation $+$ of $Sing_n(\A^m_k)(U)$ is defined point-wise and using group structure of $\A^m_k$. Let $f,g \in Sing_n(\A^m_k)(U)=Hom(\Delta^n_a\times U,\A^m_k)$, then $(f+g) \in Hom (\Delta^n_a\times U,\A^m_k)$ defined as \[ (f+g)(x_0,\dots,x_n)= f(x_0,\dots,x_n)+g(x_0,\dots,x_n) \] where the second + is the group operation of the variety $\A^m_k$. It has $\mathrm{0}\in Hom(\Delta^n_a\times U,\A^m_k)$ constant map to $(0,\dots,0)$ as identity element and the inverse of $f$ is $-f$ which is defined pointwise.  Similarly the boundary and degeneracy maps are group homomorphism. Let $f,g \in Sing_n(\A^m_k)(U)=Hom(\Delta^n_a\times U,\A^m_k)$ then 
\begin{flalign*}
	&d_i(f+g)(x_0,\dots,x_{n-1})&\\
	&= (f+g)(x_0,\dots,x_{i-1},0,x_{i},\dots,x_{n-1})&\\
	&= f(x_0,\dots,x_{i-1},0,x_{i},\dots,x_{n-1})+ g(x_0,\dots,x_{i-1},0,x_{i},\dots,x_{n-1})&\\
	&=d_i(f)+d_i(g)&	
\end{flalign*}

thus $d_i:Hom(\Delta^n_a \times U,\A^m_k)\to Hom(\Delta^{n-1}_a \times U,\A^m_k)$ is a group homomorphism. Likewise,
 \begin{align*}
 	&s_i(f+g)(x_0,\dots,x_{n+1})\\
 	&= (f+g)(x_0,\dots,x_{i-1},x_{i}+x_{i+1},\dots,x_{n+1})\\
 	&= f(x_0,\dots,x_{i-1},x_{i}+x_{i+1},\dots,x_{n+1})+ g(x_0,\dots,x_{i-1},x_{i}+x_{i+1},\dots,x_{n+1})\\
 	&=s_i(f)+s_i(g)	
 \end{align*} 
 thus $s_i:Hom(\Delta^n_a \times U,\A^m_k)\to Hom(\Delta^{n+1}_a \times U,\A^m_k)$ is a group homomorphism.
 
\begin{point}[Formula of horn filling of $Sing_*(\mathbb{A}^m_k)(U),\ U\in Sm/k$]\label{full hornfill An}~\\
	Let $f_0,\dots,f_{l-1},f_{l+1},f_n \in Sing_{n-1}(\A^m_k)(U)$ satisfying the compatibility condition $d_i(f_j)=d_{j-1}(f_i)$ for $i<j\neq l$ be a horn. Using Lemma \ref{horn-formula} and the group structure of $Sing_*(\A^m_k)(U)$ we define inductively -
	\begin{align*}
		g_0 &= s_0(f_0)& \\
		g_i &= g_{i-1}- s_i(d_i(g_{i-1})) + s_i(f_i)\  &\mathrm{ for}\ 0<i<l\\
		g_i &= g_{i+1} - s_{i-1}(d_i(g_{i+1})) + s_i(f_i)\  & \mathrm{for}\ l<i<n\\
		g_n &= g_{l-1} - s_{n-1}(d_n(g_{l-1})) + s_{n-1}(f_n).&
	\end{align*}
Similarly for $l=0$ we take $g_{-1}=0\in Hom(\Delta^n_a \times U,\A^m)$, in that case $g_n=s_{n-1}(f_n)$. Then $g_{l+1}\in Sing_n(\A^m_k)(U)$ satisfies $d_j(g_{l+1})=f_j \mathrm{\ for\ } 0\leq j\leq n$. For algebraic description of the horn filling see \cite[7.1.1]{Biman-Thesis}.

\end{point}

\

\

\

\

\bibliographystyle{alpha}
	
    \end{document}